\documentclass[a4paper,11pt]{article}
\usepackage[latin1]{inputenc}
\usepackage[T1]{fontenc}
\usepackage{a4wide}
\usepackage{amsfonts}
\usepackage{amsmath}
\usepackage{enumerate}
\usepackage{amssymb}
\usepackage{hyperref}
\usepackage{amscd}
\usepackage{amsthm}
\usepackage{fontenc}
\usepackage{cancel}
\usepackage[pdftex]{graphicx}
\usepackage{xcolor}
\usepackage{mathrsfs}
\usepackage{color}

\newcommand{\ls}{\leqslant}
\newcommand{\gs}{\geqslant}
\renewcommand{\leq}{\leqslant}
\renewcommand{\leq}{\leqslant}

\newcommand{\R}{\mathbb R}

\newcommand{\eps}{\varepsilon}
\newtheorem{theo}{Theorem}[section]
\newtheorem{defi}{Definition}[section]
\newtheorem{prop}{Proposition}[section]
\newtheorem{cor}{Corollary}[section]
\newtheorem{lemma}{Lemma}[section]
\newtheorem{rem}{Remark}[section]

\numberwithin{equation}{section}

\title{Anisotropic curvature flow of immersed networks}

\date{\today}

\author{
Heiko Kr\"oner\footnote{Universit\"at Duisburg-Essen,
Fakult\"at f\"ur Mathematik, Essen, Germany,
heiko.kroener@uni-due.de},
Matteo Novaga\footnote{Dipartimento di Matematica, 
Universit\`a di Pisa, Pisa, Italy, 
matteo.novaga@unipi.it},
Paola Pozzi\footnote{Universit\"at Duisburg-Essen,
Fakult\"at f\"ur Mathematik, Essen, Germany,
paola.pozzi@uni-due.de}
}

\begin{document}
\maketitle

\begin{abstract}

We consider  motion by anisotropic curvature of a network of three curves immersed in the plane meeting at a triple junction and with the other ends fixed.
We show existence, uniqueness and regularity of a maximal geometric solution and we prove that, if the maximal time is finite, then either 
the length of one of the curves goes to zero or the $L^2$ norm of the anisotropic curvature blows up.

\end{abstract}

\tableofcontents

\section{Introduction}

The aim of this work is to study  motion by anisotropic curvature of a network of three curves in the plane. 
This evolution corresponds to a gradient flow of the anisotropic length of the network, which is the sum of the anisotropic lengths of the three curves.
Since multiple points of order greater than three are always energetically unstable (see \cite{Taylor, LAMO}), it is natural to consider 
networks with only triple junctions, the simplest of which is a network with three curves meeting at a common point.

The isotropic version of this problem attracted a considerable attention in recent years (see for instance  the extended survey \cite{MNPS} and references therein). In particular, the short time existence for the evolution has been first proved by L. Bronsard and F. Reitich in \cite{BR}, 
and later extended in  \cite{MNT,MMN}
where it is shown that, at the maximal existence time, either one curve disappears or the curvature blows up.

The main result of this paper, contained in Theorem \ref{teo:sciop}, is the extension of the result in \cite{MNT} to the smooth anisotropic setting.
More precisely, we show that at the maximal existence time of the geometric solution (see Definition~\ref{geomsol}), 
either the length of one curve goes to zero or the $L^2$ norm of the anisotropic curvature blows up. 
In the latter case, we also provide a lower bound on the  blow up rate of the curvature (see Lemma~\ref{lem5.2}).

A relevant technical issue in this paper is due to the fact that,
in the case of networks, the evolution is governed by a system of PDE's rather than by a single equation, 
hence it is difficult to use the maximum principle, which is usually the main tool to get estimates on the geometric quantities for curvature flows.
As a consequence,  following \cite{MNT} in order  to control these quantities we rely on
delicate integral estimates and interpolation inequalities.

A challenging open problem is the extension of such result to the nonsmooth (including crystalline) anisotropic setting, 
as it was done in \cite{chambolle13,MNP} for the case of closed planar curves.
In the case of networks, the dependence of the  integral estimates on the anisotropy, makes such extension problematic.

Let us point out that, in the paper  \cite{BCN}, the authors proved a short time existence result for the crystalline evolution of embedded networks,
under a suitable assumption on the initial data which allows to reduce the evolution equation to a system of ODE's.
We also recall that in the papers \cite{KT,BCK} the authors discuss 
existence of global weak solutions for the evolutions of embedded networks by anisotropic curvature flow.

\smallskip

The paper is organized as follows: In Section \ref{sec1} we introduce the notation and define the relevant geometric object that we shall use throughout the paper. In Section \ref{sec:2} we prove a short time existence result for the evolution following the approach in \cite{BR,MNT}.
In Section \ref{sec:maxsolGEOP} we show the existence and uniqueness of a maximal geometric solution and we prove that, at the maximal time, either the length of one curve tends to zero or the $H^1$ norm of the anisotropic curvature blows up.
Finally, in Section \ref{sec:integralEst} we refine this conclusion by showing that, if the $H^1$ norm  blows up, then also the $L^2$ norm  
of the anisotropic curvature blows up.
We conclude the paper with an Appendix containing some technical result which are used in the paper.

\smallskip

\noindent\textbf{Acknowledgments:} MN is a member of the INDAM/GNAMPA and
acknowledges partial support by the PRIN 2017 Project \emph{Variational methods for stationary and evolution problems with singularities and interfaces}. HK und PP have been supported by the DFG (German Research Foundation) Projektnummer: 404870139.

\section{Notation and preliminary definitions}\label{sec1}

We consider regular planar curves parametrized by $u: [0, T]\times I \to \R^2$, where $I=[0,1]$. 
We denote by $s$ the arc-length parameter of the curve (thus $\partial_s (\cdot)=\partial_x (\cdot)/ |u_x|$), by $\tau =u_x/|u_x|=u_s=(\sin \theta, -\cos \theta)$ its unit tangent and $\nu =(\cos \theta, \sin \theta)$ its unit normal.
The Euclidean scalar product in $\R^2$ is denoted by $\cdot$. The symbol $\perp $  stands for anti-clockwise rotation by $\pi/2$, therefore $(a,b)^\perp= (-b, a)$.
 Recall the classical Frenet formulas
\begin{align}
u_{ss}=\tau_s= \vec{\kappa}= \kappa \nu, \qquad \nu_s =-\kappa \tau.
\end{align}
Obviously $\vec{\kappa}=\frac{u_{xx}}{|u_{x}|^{2}} - \frac{u_{xx}}{|u_{x}|^{2}} \cdot \tau \tau$ and $\kappa = \frac{u_{xx}}{|u_{x}|^{2}} \cdot \nu$.
Moreover recall that  from the expression for $\nu_s$ one infers that for the scalar curvature $\kappa$ we have
\begin{align}
\kappa=\theta_s.
\end{align}

\subsubsection{Anisotropies}

Let us recall some definitions and properties of anisotropy maps (see for instance \cite{BePa}).
\begin{defi}
 We call  \emph{anisotropy} a norm $\varphi: \R^2 \to [0, \infty)$.
We say that $\varphi$  is \emph{smooth}  if $\varphi \in C^{\infty}(\R^{2} \setminus \{ 0 \})$ and $\varphi$ is  \emph{elliptic} if $\varphi^{2}$ is uniformly convex, that is, there exists $C > 0$ such that
\begin{align}\label{ell-cond}
 D^2(\varphi^{2}) \gs C \, \mbox{Id}
 \end{align}
in the distributional sense. 
\end{defi}
\begin{defi}
The set $W_{\varphi}:=\{\varphi \leq 1 \}$ is called  \emph{Wulff shape}. We say that
$\varphi$ is \emph{crystalline} if  $W_{\varphi}$ is a polygon.
\end{defi}
\begin{defi}
 Given an anisotropy $\varphi$, we introduce the \emph{polar norm} $\varphi^\circ$ relative to $\varphi$
 $$\varphi^\circ (x) = \sup\{ \xi \cdot x \, \vert \, \varphi(\xi) \ls 1\}.$$
\end{defi}

\begin{rem}\rm
Note that $\varphi$ is smooth and elliptic if and only if 
$\varphi^\circ$ is smooth and elliptic (\cite[\S~2]{chambolle13}). 

The ellipticity condition implies  that the Wulff shape is uniformly convex. 
Moreover, from \eqref{ell-cond} one infers that 
\begin{align}\label{ell-cond2}
D^{2}\varphi (\nu) \tau  \cdot \tau \geq \widetilde{C}, \qquad \widetilde{C}:=\frac{C}{2 \max  \{\varphi(\tilde{\nu}) \, | \, \tilde{\nu} \in S^{1}  \}},
\end{align}
for  unit vectors $\nu$ and $\tau$ with $\nu \cdot \tau =0$ (see \cite[Remark~1]{MNP}).
\end{rem}

In the following, we shall restrict ourselves to the case of smooth and elliptic anisotropies. 

\smallskip

Observe that the homogeneity property of a norm $\varphi$ yields $D \varphi(p) \cdot p = \varphi(p)$ and $D^{2}\varphi (p)p =0$ for any $p \neq 0$, two facts that we will use repeatedly in our computations.

\subsubsection{Anisotropic scalar curvature and anisotropic curve shortening flow}

When $u$ is smooth and the anisotropy $\varphi$ is smooth and elliptic
the classical formulation of the anisotropic curvature flow
is given by the equation  (see \cite{angenent90}) 
\begin{align}\label{acsf}
u_t= \varphi^\circ(\nu) \kappa_\varphi \nu, 
\end{align}
 where the scalar anisotropic curvature is given by
 \begin{align}
 \kappa_\varphi := -N_{s} \cdot \tau
 \end{align}
 with $N=D\varphi^{\circ}(\nu)$ the Cahn-Hoffman vector. Thus $$\kappa_\varphi = D^{2} \varphi^{\circ} (\nu) \tau \cdot \tau \kappa.$$
Clearly, boundary and initial conditions (and compatibility conditions) have to be specified as well, but for the moment  we neglelct those and focus only on the evolution equation. 
By setting
\begin{align}
  \phi (\theta):=\varphi^\circ(\nu)=\varphi^\circ(\cos \theta, \sin \theta),
\end{align}
a straightforward calculation gives
\begin{align}
\label{a1}
\phi(\theta)+ \phi''(\theta)=D^2 \varphi^\circ(\nu) \tau \cdot \tau ,
\end{align}
so that we can rewrite the  flow \eqref{acsf} as
\begin{align}
\label{ACSFF}
u_t=\phi(\theta)(\phi(\theta)+ \phi''(\theta)) \,\kappa \nu= \psi (\theta) \kappa \nu ,
\end{align}
where $\kappa$ is the Euclidean curvature and
\begin{align} \label{defpsi}
\psi(\theta):=\phi(\theta)(\phi(\theta)+ \phi''(\theta))  = \varphi^\circ(\nu) D^2 \varphi^\circ(\nu) \tau \cdot \tau .
\end{align}
Note that by \eqref{ell-cond2}, the ellipticity of $\varphi$ implies uniform bounds for $\psi$, i.e.,  
 \begin{align}\label{boundmpiccolo}
  M \geq \psi \geq m>0.
\end{align}

In the following we shall admit tangential components to the flow, therefore  we will  consider evolution equations of type 
\begin{align}\label{acsfT}
u_t= \varphi^\circ(\nu) \kappa_\varphi \nu + \lambda \tau =\psi(\theta) \kappa \nu + \lambda \tau,
\end{align}
for some sufficiently smooth scalar function $\lambda$.
\begin{defi}
The \textbf{special anisotropic curve shortening flow} is defined through a specific choice of tangential term, namely
we take $\lambda= \varphi^{\circ}(\nu) (D^2 \varphi^\circ(\nu) \tau \cdot \tau) \frac{u_{xx}}{|u_{x}|^{2}} \cdot \tau$  in \eqref{acsfT}. Thus, the special anisotropic curve shortening flow is given by
\begin{align}\label{SpecialFlow}
u_{t}=  \varphi^{\circ}(\nu) (D^2 \varphi^\circ(\nu) \tau \cdot \tau)  \frac{u_{xx}}{|u_{x}|^{2}} =\psi(\theta) \frac{u_{xx}}{|u_{x}|^{2}} .
\end{align}
\end{defi}

Next we derive the evolution laws of  relevant geometric quantities.
\begin{lemma}
\label{lemma2.1}
Assume $u$ satisfies \eqref{acsfT}. Then,
the following equalities hold
\begin{align}\nonumber
\partial_t \partial_s (\cdot) &= \partial_s \partial_t (\cdot) +\psi(\theta) \kappa^2 \partial_s (\cdot) -\lambda_{s} \partial_{s} (\cdot)
\\\nonumber
\tau_t &= [(\psi(\theta) \kappa)_s + \lambda \kappa ]\nu
\\\nonumber
\nu_t & = -[(\psi(\theta) \kappa)_s +\lambda \kappa ] \tau 
\\
\label{k_t}
\kappa_t &= (\psi(\theta) \kappa)_{ss} + \psi(\theta) \kappa^3  + \lambda \kappa_{s}
\\\nonumber
\theta_t &=(\psi(\theta) \kappa)_s  + \lambda \kappa. 
\end{align}
For the special flow \eqref{SpecialFlow} where $\lambda =\psi(\theta) \frac{u_{xx}}{|u_{x}|^{2}} \cdot \tau = - \psi(\theta) \frac{\partial}{\partial x} \left( \frac{1}{|u_{x}|} \right)$ we have that
\begin{align}\label{l_{t}}
\lambda_{t} =\frac{\lambda}{\psi(\theta)} \psi'(\theta) [(\psi(\theta) \kappa)_s  + \lambda \kappa] + \psi(\theta) \lambda_{ss} - \psi(\theta) (\psi(\theta) \kappa^{2})_{s} -\lambda \lambda_{s} + \lambda \psi(\theta) \kappa^{2}.
\end{align}
\end{lemma}
\begin{proof}
The assertions easily follow by straightforward calculations, see for instance \cite[Lemma~1]{MNP} for the special case where $\lambda=0$ and \cite[Lemma~3.1]{MNPS} for the  isotropic case.
\end{proof}

\begin{lemma}
Assume $u$ satisfies \eqref{acsfT}. 
Then  the following holds for the isotropic and anisotropic length of the curve
\begin{align}
\frac{d}{dt} L(u) &=\frac{d}{dt} \int_{I} ds=-\int_{I} \psi(\theta) \kappa^{2} ds + [\lambda]_{0}^{1}, \\ \label{dtlan}
\frac{d}{dt} L_{\varphi}(u) & :=\frac{d}{dt} \int_{I} \varphi^{\circ} (\nu) ds = - \int_{I} \kappa_{\varphi}^{2} \varphi^{\circ}(\nu) ds + [\varphi^{\circ}(\nu) \lambda - \psi(\theta) \kappa D \varphi^{\circ}(\nu) \cdot \tau]_{0}^{1}.
\end{align}
\end{lemma}
\begin{proof} We compute
\begin{align*}
\frac{d}{dt} L(u) & 
= \int_{I} \tau \cdot u_{tx} dx =
-\int_{I}  \kappa \nu \cdot \psi(\theta) \kappa \nu ds + \int_{I} \tau \cdot (\lambda \tau)_{x} dx 
= -\int_{I} \psi(\theta) \kappa^{2} ds + [\lambda]_{0}^{1}
\end{align*}
and
\begin{align*}
\frac{d}{dt} L_{\varphi}(u) & =\frac{d}{dt} \int_{I} \varphi^{\circ} (\nu) ds = \int_{I} D \varphi^{\circ}(\nu) \cdot u_{tx}^{\perp} dx = \int_{I} D \varphi^{\circ}(\nu) \cdot ( - \psi(\theta) \kappa \tau + \lambda \nu)_{x} dx\\
&=- \int_{I} D^{2} \varphi^{\circ} (\nu) \tau \cdot \tau \psi(\theta) \kappa^{2} ds + [\varphi^{\circ}(\nu) \lambda - \psi(\theta) \kappa D \varphi^{\circ}(\nu) \cdot \tau]_{0}^{1}\\
&=- \int_{I}  \frac{ \psi(\theta)^{2}}{\varphi^{\circ} (\nu)} \kappa^{2} ds + [\varphi^{\circ}(\nu) \lambda - \psi(\theta) \kappa D \varphi^{\circ}(\nu) \cdot \tau]_{0}^{1}.
\end{align*}
\end{proof}

\subsubsection{The Geometric Problem}\label{sec:geoprob}

For basic definitions of networks see for instance \cite[\S~2]{MNPS}.
We consider  networks  $\mathbb{S}$ of curves parametrized by regular maps $u^{i}: [0,1] \to \R^{2}$, $i=1,2,3$,
 such that $ u^{i}(1)=P^{i}$ (with $P^{i} \in \R^{2}$ given) and $u^{i}(0)=u^{j}(0)$, for $i, j \in \{ 1,2,3 \}$, that is the curves are  parametrized in such a way that the origin is mapped to the triple junction.

\begin{figure}[h]
    \centering
     \includegraphics[width=0.25\textwidth]{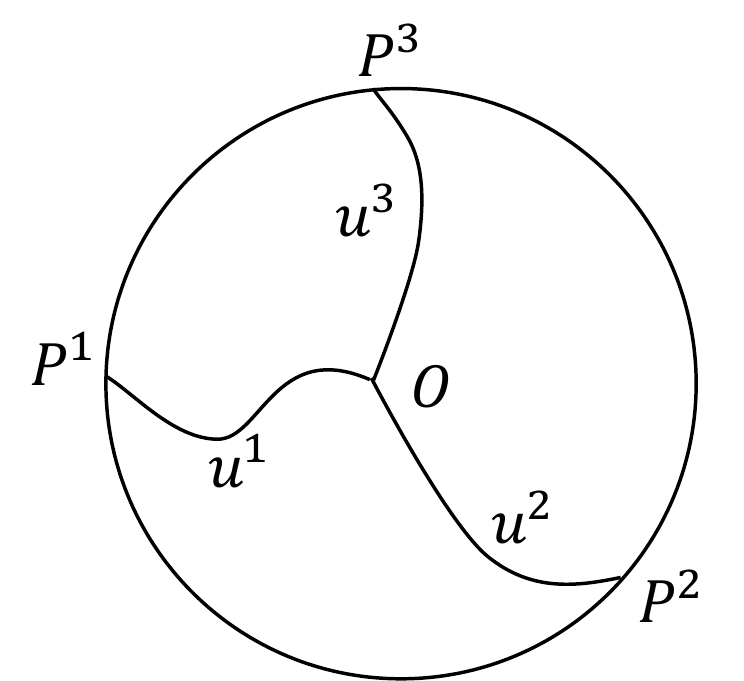}
    \caption{Network with one triple point $O$ and three endpoints $P^1,P^2,P^3$}
    \label{figure1}
\end{figure}

\begin{defi}[Geometrically admissible networks] \label{admtriod}
A network $\mathbb{S}$  is called admissible if there exist regular parametrizations
$\sigma^{i} \in C^{2, \alpha}([0,1], \R^{2})$, $i=1,2,3$ such that $\mathbb{S} = \cup_{i=1}^{3}\sigma^{i} ([0,1])$ and  there holds
\begin{align*}
\left\{ \begin{array}{lr}
\sigma^{i}(1) =P^{i} &    i=1,2,3,\\
\sigma^{1}(0)=\sigma^{2}(0)=\sigma^{3}(0) &  \\
\sum_{i=1}^{3} D \varphi^{\circ} (\nu^{i}_{0}) =0 &  \text{ where  } \nu^{i}_{0}:= \frac{(\sigma^{i}_{x})^{\perp}}{|\sigma^{i}_{x}|},
\end{array}
\right.
\end{align*}
together with 
$$ \kappa_{\varphi}^{i}=0 \quad \text{ at } x=1,$$
(where $\kappa_{\varphi}^{i}$ denotes the anisotropic curvature of the curve $\sigma^{i}$) and
$$ 
 \kappa_{\varphi}^{i} \varphi^{\circ}(\nu^{i}_{0}) \nu^{i}_{0} +\lambda^{i}_{0} \tau^{i}_{0}=\kappa_{\varphi}^{j}\varphi^{\circ}(\nu^{j}_{0})\nu^{j}_{0}+\lambda^{j}_{0}\tau^{j}_{0} \quad \text{ for } i,j \in \{ 1,2,3 \} \text{ at } x=0.$$
Here $\lambda^{i}_{0}$ denotes a further geometric quantity,
 whose expression is formulated in \eqref{expressionlambda} below.
 In particular we see that $\lambda^{i}_{0}$ is given as a linear combination of $\psi(\theta^{i}) \kappa^{i}$ and $\psi(\theta^{i\pm 1}) \kappa^{i\pm1}$.
\end{defi}


\begin{defi}\label{geomsol}
 Given an initial admissible network $\sigma:=( \sigma^{1}, \sigma^{2}, \sigma^{3})$  as in Definition~\ref{admtriod} 
we look for $T>0$ and regular maps $u^{i}: [0,T) \times [0,1] \to \R^{2}$, $i=1,2,3$, 
with $u^{i} \in C^{\frac{2+\alpha}{2}, 2+\alpha} ([0,T) \times [0,1], \R^{2}) $
such that
\begin{align}
 (u^{i}_{t} \cdot \nu^{i}) \nu^{i} = \psi(\theta^{i} ) \kappa^{i} \nu^{i}  \qquad \text{ on } (0,T) \times (0,1)  \qquad \qquad i=1,2,3,
\end{align}
with initial datum $u^{i}(0, \cdot)=\sigma^{i}(\cdot)$ up to reparametrization ( i.e.,
$u^{i}(0, \cdot)=\sigma^{i}(\phi^{i}(\cdot))$ for some orientation preserving diffeomorphism $\phi^{i}\in C^{2,\alpha} ( [0,1] ,[0,1])$ 
) and (natural) boundary conditions
\begin{align}
\left\{ \begin{array}{lr}
u^{i}(t,1) =P^{i} &  \text{ for all } t \in (0,T), \, i=1,2,3,\\
u^{1}(t,0)=u^{2}(t,0)=u^{3}(t, 0) &  \text{ for all } t \in (0,T),\\
\sum_{i=1}^{3} D \varphi^{\circ} (\nu^{i}(t,0)) =0 &  \text{ for all } t \in (0,T).
\end{array}
\right.
\end{align}
A solution to such problem is called \textbf{geometric solution}.
\end{defi}

\begin{rem}[Anisotropic angle condition]
The boundary condition
\begin{align}\label{HC}
\sum_{i=1}^{3} D \varphi^{ \circ} (\nu^{i}) =0
\end{align}
at the triple junction is the anisotropic version of the Herring condition (cf. \cite[Def. 2.5]{MNPS}) and is derived by considering the first variation of
$
E(\mathbb{S}) := \sum_{i=1}^{3} \int_{I} \varphi^{\circ} (\nu^{i}) ds^{i}.
$
Indeed, for variations of type $u^{i}+ \epsilon \varphi^{i}$, where $\varphi^{i}$ are smooth functions with $\varphi^{i}(1)=0$, $\varphi^{i}(0)=\varphi^{j}(0)$ for all $i, j \in \{ 1,2,3 \}$ we can write
\begin{align*}
\frac{d}{d\epsilon} E(\mathbb{S}_{\epsilon})= - \sum_{i=1}^{3} \int_{[0,1]} (D^{2} \varphi^{ \circ} (\nu^{i}) \tau^{i} \cdot \tau^{i} ) \kappa^{i} \nu^{i} \cdot \varphi^{i} ds - \sum_{i=1}^{3} D \varphi^{ \circ} (\nu^{i}(0))  \cdot (\varphi^{i}(0))^{\perp},
\end{align*}
(where here and in the following we write $ds$ instead of $ds^{i}$, the meaning being clear from the context)
and \eqref{HC} is immediately deduced.
Note that the vectors $\xi^{i}:= D \varphi^{\circ}(\nu^{i})$  appearing in \eqref{HC} belong to the boundary of the Wulff shape, i.e., $\xi^{i} \in \partial W_{\varphi}$, $i=1,2,3$. 
We can state that the angles at which the tangent planes to $\partial W_{\varphi}$ at $\xi^{i}$ can meet are  bounded away from zero and $\pi$: indeed in one of these two limit cases, the three vectors must be in shape of a Y (possibly with two vectors coinciding), but 
 we get a contradiction using the symmetry and  convexity of the Wulff shape.

Since $ \nu^{i}$ is normal to the tangent plane at $\xi^{i} =D \varphi^{\circ}(\nu^{i}) \in \partial W_{\varphi}$, this means that there exists a positive constant $ C$ depending on $\varphi^{\circ}$ such that
\begin{align*}
0 \leq | \nu^{i} \cdot \nu^{j}| \leq C <1, \qquad i \neq j, \quad (i, j \in \{ 1,2,3 \}).
\end{align*}
In turns this implies  the existence of a postive constant $ a_{0}$ depending on $\varphi^{\circ}$ such that
\begin{align}\label{a0}
|\nu^{i} \cdot \tau^{j}| \geq a_{0} >0 \qquad i \neq j, \quad (i, j \in \{ 1,2,3 \}).
\end{align}
\end{rem}

For the notion of geometric solution it is enough to specify the normal velocity. To   attack the problem analytically, we actually consider the system
\begin{align}
 u^{i}_{t}  = \psi(\theta^{i} ) \kappa^{i} \nu^{i} + \lambda^{i} \tau^{i} \qquad \qquad i=1,2,3,
\end{align}
for some  scalar maps $\lambda^{i} \in C^{\frac{\alpha}{2}, \alpha} ([0,T) \times [0,1], \R^{2})$.
Note that the presence of  tangential components $\lambda^{i}$ is necessary to allow for movements of the triple junction.
In principle there is some freedom in the choice of these maps, but the freedom is restricted only to the points in the interior of the interval of definition.
Indeed we show below in Section~\ref{behav} that $\lambda^{i}$, $i=1,2,3$ are fixed by the problem  at the boundary. More precisely we show that at the boundary we can express $\lambda^{i}$ as a linear combination of the geometric quantities  $\psi(\theta^{i}) \kappa^{i}$ and $\psi(\theta^{i\pm 1}) \kappa^{i\pm1}$.

Among all possible choices of tangential components $\lambda^{i}$, we highlight one specific flow that will play
an important role in our discussion:
\begin{defi}\label{def:SF}
A  solution as in Definition~\ref{geomsol} such that $u_{t}^{i}$, $i=1,2,3$, evolves according to \eqref{SpecialFlow}  is called \textbf{Special Flow}.
\end{defi}
The Special Flow provides a well posed problem that we can attack analytically.
We shall use the Special Flow to derive short time-existence of  a geometric solution, and to  show its uniqueness and smoothness.

 
\subsubsection{Behavior of a generic tangential component $\lambda^{i}$ at the triple junction}\label{behav}
At the triple junction beside the concurrency condition we impose that the velocity  be the same for all curves involved, hence we impose
\begin{align}\label{velcond}
\psi(\theta^{i}) \kappa^{i} \nu^{i} + \lambda^{i} \tau^{i} =\psi(\theta^{j}) \kappa^{j} \nu^{j} + \lambda^{j} \tau^{j}
\end{align}
or equivalently (after rotation by $\pi/2$)
\begin{align*}
-\psi(\theta^{i}) \kappa^{i} \tau^{i} + \lambda^{i} \nu^{i} =-\psi(\theta^{j}) \kappa^{j} \tau^{j} + \lambda^{j} \nu^{j}
\end{align*}
for every $i,j  \in \{ 1,2,3 \}$.
 Multiplying with $D \varphi^{\circ}(\nu^{i})$, summing over $i$, and using \eqref{HC} gives
 \begin{align}
 0= \sum_{i=1}^{3} \psi(\theta^{i}) \varphi^{\circ}(\nu^{i}) \kappa^{i} + \lambda^{i} (\tau^{i} \cdot D \varphi^{\circ} (\nu^{i}))
\end{align}
and
\begin{align}\label{22}
 0= \sum_{i=1}^{3}  \varphi^{\circ}(\nu^{i}) \lambda^{i} - \psi(\theta^{i})\kappa^{i}(\tau^{i} \cdot D \varphi^{\circ} (\nu^{i})).
 \end{align}
 In the isotropic case this amounts to $\sum_{i=1}^{3} \kappa^{i} =0 = \sum_{i=1}^{3} \lambda^{i}$.
 
 On the other hand, starting from \eqref{velcond} and taking the inner product with appropriate normals and tangents we get (with the convention that the superscripts are considered ``modulus 3'')
 \begin{align*}
 \psi(\theta^{i}) \kappa^{i}  = \psi(\theta^{i\pm 1}) \kappa^{i\pm 1} (\nu^{i\pm 1} \cdot \nu^{i}) + \lambda^{i\pm 1} (\tau^{i\pm 1} \cdot \nu^{i}) ,\\
  \lambda^{i}  =\psi(\theta^{i\pm 1}) \kappa^{i\pm 1} (\nu^{i\pm 1} \cdot \tau^{i}) + \lambda^{i\pm 1} (\tau^{i\pm 1} \cdot \tau^{i}).
 \end{align*}For the isotropic case where all constants and coefficients can be given explicitly see \cite[\S3]{MNPS}.
 The above system can be written as
 \begin{align*}
 \left( \begin{array}{cccc}
 (\nu^{i+1} \cdot \nu^{i}) & 0 & (\tau^{i+1} \cdot \nu^{i}) & 0\\
 0 & (\nu^{i-1} \cdot \nu^{i})& 0 &  (\tau^{i-1} \cdot \nu^{i}) \\
 (\nu^{i+1} \cdot \tau^{i}) & 0 & (\tau^{i+1} \cdot \tau^{i})& 0 \\
 0& (\nu^{i-1} \cdot \tau^{i}) & 0 & (\tau^{i-1} \cdot \tau^{i})
 \end{array} \right )
 \left ( \begin{array}{c}
 \psi(\theta^{i+1}) \kappa^{i+1}\\
 \psi(\theta^{i-1}) \kappa^{i-1}\\
 \lambda^{i+1}\\
 \lambda^{i-1}
 \end{array}
 \right) =
 \left ( \begin{array}{c}
 \psi(\theta^{i}) \kappa^{i}\\
 \psi(\theta^{i}) \kappa^{i}\\
 \lambda^{i}\\
 \lambda^{i}
 \end{array}
 \right) .
 \end{align*}
 Writing $\alpha = (\nu^{i+1} \cdot \nu^{i})$, $ \beta= (\tau^{i+1} \cdot \nu^{i})$, $\gamma=(\nu^{i-1} \cdot \nu^{i})$, $\delta=(\tau^{i-1} \cdot \nu^{i})$  we see that 
 above matrix has determinant equal to $det=(\alpha^{2}+\beta^{2})(\delta^{2}+\gamma^{2})$, which can never be zero  since $\alpha$ and $\beta$, respectively $\delta$ and $\gamma$, can not vanish simultaneously. Thus  we obtain
 \begin{align*}
 \left ( \begin{array}{c}
 \psi(\theta^{i+1}) \kappa^{i+1}\\
 \psi(\theta^{i-1}) \kappa^{i-1}\\
 \lambda^{i+1}\\
 \lambda^{i-1}
 \end{array}
 \right) =
 \left( \begin{array}{cccc}
 \frac{\alpha}{\alpha^{2} + \beta^{2}} & 0 & -\frac{\beta}{\alpha^{2} + \beta^{2}} & 0\\
 0 & \frac{\gamma}{\gamma^{2} +\delta^{2}}& 0 &  -\frac{\delta}{\gamma^{2} +\delta^{2}} \\
 \frac{\beta}{\alpha^{2} + \beta^{2}} & 0 & \frac{\alpha}{\alpha^{2} + \beta^{2}}& 0 \\
 0& \frac{\delta}{\gamma^{2} +\delta^{2}} & 0 & \frac{\gamma}{\gamma^{2} +\delta^{2}}
 \end{array} \right )
 \left ( \begin{array}{c}
 \psi(\theta^{i}) \kappa^{i}\\
 \psi(\theta^{i}) \kappa^{i}\\
 \lambda^{i}\\
 \lambda^{i}
 \end{array}
 \right) .
 \end{align*}
 From the first two equations we infer that if $\beta\neq0$ or $\delta\neq 0$ then we can express $\lambda^{i}$ as a linear combination of $\psi(\theta^{i}) \kappa^{i}$ and $\psi(\theta^{i\pm 1}) \kappa^{i\pm1}$.
 By \eqref{a0} we know  that in fact  $|\beta|$ and $|\delta|$ are bounded from below.
In particular we obtain that
\begin{align}\label{expressionlambda}
\lambda^{i}= \frac{\alpha}{\beta}\psi(\theta^{i}) \kappa^{i} -\frac{1}{\beta} \psi(\theta^{i+1}) \kappa^{i+1} \quad \text{ and } \quad  \lambda^{i}= \frac{\gamma}{\delta}\psi(\theta^{i}) \kappa^{i} -\frac{1}{\delta} \psi(\theta^{i-1}) \kappa^{i-1},
\end{align}
so that
\begin{align}\label{boundlbdry}
|\lambda^{i}| \leq C \sum_{j=1}^{3}|\psi(\theta^{j})\kappa^{j}| \qquad \qquad i=1,2,3
\end{align}
at the triple junction with $C=C(a_{0})$ depending on the anisotropy. 

For the analysis that follows we will also need  expressions for  the time derivative $\lambda_{t}^{i}$.  Using \eqref{expressionlambda} we can write
\begin{align}\label{derla}
|\lambda_{t}^{i}|  \leq \left|\left(\frac{\alpha}{\beta} \right)_{t} \right| | \psi(\theta^{i}) \kappa^{i}| +
C |( \psi(\theta^{i}) \kappa^{i})_{t}| +\left|\left(\frac{1}{\beta} \right)_{t} \right| | \psi(\theta^{i+1}) \kappa^{i+1}| +
C |( \psi(\theta^{i+1}) \kappa^{i+1})_{t}|
\end{align}
with $C=C(a_{0})$ depending on the anisotropy. 
\begin{lemma} The total anisotropic length of the network decreases in time along the evolution.
\end{lemma}
\begin{proof}
The statement follows by adding the contribution of each curve as computed in \eqref{dtlan}, using \eqref{22} at the triple junction, and the fact that $\lambda^{i}=0=\kappa^{i}$ at the fixed points $P^{i}$, $i=1,2,3$ (this follows from \eqref{acsfT} and $\partial_{t} u^{i} =0$ at $P_{i}$).
\end{proof}

\subsubsection{Special Flow: behavior of $\lambda^{i}$ in the interior points}
In the following we assume that \eqref{SpecialFlow} holds for every curve of the network and that we have a uniform bound on the curvatures, namely
\begin{align*}
\sum_{i=1}^{3}\sup_{t \in [0,T]} \| \kappa^{i}(t, \cdot)\|_{L^{\infty}} \leq C_{0}.
\end{align*}
 Since the following considerations hold for any curve of the network we drop the indices for simplicity  of notation.
Upon recalling \eqref{acsfT} let us denote with $V$ the length of the velocity vector. Then 
\begin{align}
V^{2}=|u_{t}|^{2} = (\psi(\theta)\kappa)^{2} +\lambda^{2}
\end{align}
Using Lemma~\ref{lemma2.1} (in particular also \eqref{l_{t}}) we observe that $w:=V^{2}$ satisfies (cp. with \cite[page~263]{MNT} for the isotropic case)
\begin{align*}
w_{t}&= \psi(\theta) w_{ss} -\lambda w_{s} + 2 \psi(\theta) \kappa^{2} w 
-2 \psi(\theta) [(\psi(\theta) \kappa)_{s}]^{2} -2 \psi(\theta) (\lambda_{s})^{2} + N
\end{align*}
where
\begin{align*}
N&= 2 ((\psi(\theta) \kappa)_{s} +\lambda \kappa) \left( (\psi(\theta) \kappa) (\psi(\theta))_{s}+ \lambda^{2}
\frac{\psi'(\theta)}{\psi(\theta)}\right) = 2 ((\psi(\theta) \kappa)_{s} +\lambda \kappa) \frac{\psi'(\theta)}{\psi(\theta)} ((\psi(\theta)\kappa)^{2} +\lambda^{2})\\
&=2 ((\psi(\theta) \kappa)_{s} +\lambda \kappa) \frac{\psi'(\theta)}{\psi(\theta)} w
=2 \theta_{t} \frac{\psi'(\theta)}{\psi(\theta)} w = 2(\ln \psi(\theta))_{t} w.
\end{align*}
Note that $N$ vanishes in the isotropic case. Bringing $N$ to the left-hand side and multiplying both side of the equation with $e^{-2 \ln  \psi(\theta)}$
we obtain
\begin{align*}
(w e^{-2 \ln  \psi(\theta)})_{t}&=\psi(\theta) e^{-2 \ln  \psi(\theta)} w_{ss} -\lambda e^{-2 \ln  \psi(\theta)} w_{s} + 2 \psi(\theta) \kappa^{2} w e^{-2 \ln  \psi(\theta)}\\
&\quad - \big(2 \psi(\theta) [(\psi(\theta) \kappa)_{s}]^{2} +2 \psi(\theta) (\lambda_{s})^{2} \big) e^{-2 \ln  \psi(\theta)}
\end{align*}
If $w(t, \cdot)=V^{2}(t, \cdot) \geq0$ does not take its maximum at the boundary (where $\kappa$ and hence $\lambda$, recall \eqref{boundlbdry}, are controlled by assumption) then
it achieves its maximum $w_{max}(t)= \max_{[0,1]} w(t, \cdot)$ in an interior point. By Hamilton' trick (\cite[Lemma~2.1.3]{mantegazza}) we have that $\frac{\partial}{\partial t}w_{max}(t) = w_{t} (t,x_{max}) $ where $x_{max} \in(0,1)$ is an interior point where $w(t, \cdot)$ assumes its maximum. Then
\begin{align*}
(w_{max} e^{-2 \ln  \psi(\theta)})_{t} \leq 2 \psi(\theta) \kappa^{2} w_{max} e^{-2 \ln  \psi(\theta)} \leq C w_{max} e^{-2 \ln  \psi(\theta)}
\end{align*}
where $C$ depends on $C_{0}$ and on the anisotropy map (recall \eqref{boundmpiccolo}).  Gronwall's inequality yields
\begin{align*}
w_{max} e^{-2 \ln  \psi(\theta)} \leq  e^{CT} ( w_{max} e^{-2 \ln  \psi(\theta)})|_{t=0}.
\end{align*}
It follows that $V^{i}$ and $\lambda^{i}$ are uniformly bounded on $[0,T)$ for $i=1,2,3$.

\section{Short-time existence for the Special Flow }
\label{sec:2}
The aim of this section is to establish a short time existence result for the special anisotropic curve shortening flow (recall Definition~\ref{def:SF} and \eqref{SpecialFlow}). More precisely, given an initial network $\sigma:=( \sigma^{1}, \sigma^{2}, \sigma^{3})$ of sufficiently smooth regular curves satisfying appropriate boundary conditions (see below) we look for $T>0$ and $u^{i}: [0,T] \times [0,1] \to \R^{2}$, $u^{i} \in C^{\frac{2+\alpha}{2}, 2+\alpha}([0,T] \times [0,1])$, $i=1,2,3$, $\alpha \in (0,1)$ such that
\begin{align}\label{OP1}
 u^{i}_{t} = \psi(\theta^{i} )\frac{u^{i}_{xx}}{|u^{i}_{x}|^{2}}=\varphi^{\circ}(\nu^{i}) (D^2 \varphi^\circ(\nu^{i}) \tau^{i} \cdot \tau^{i})  \frac{u_{xx}^{i}}{|u_{x}^{i}|^{2}}  \qquad \qquad i=1,2,3,
\end{align}
with initial datum $u^{i}(0, \cdot)=\sigma^{i}(\cdot)$ and boundary conditions
\begin{align}\label{OP2}
\left\{ \begin{array}{lr}
u^{i}(t,1) =P^{i} &  \text{ for all } t \in [0,T], \, i=1,2,3,\\
u^{1}(t,0)=u^{2}(t,0)=u^{3}(t, 0) &  \text{ for all } t \in [0,T],\\
\sum_{i=1}^{3} D \varphi^{\circ} (\nu^{i}(t,0)) =0 &  \text{ for all } t \in [0,T].
\end{array}
\right.
\end{align}

We assume that $\sigma^{i} \in C^{2, \alpha}([0,1], \R^{2})$, $i=1,2,3$, are regular maps fulfilling the following compatibility conditions:
\begin{align}\label{cc1}
\left\{ \begin{array}{lr}
\sigma^{i}(1) =P^{i} &    i=1,2,3,\\
\sigma^{1}(0)=\sigma^{2}(0)=\sigma^{3}(0) &  \\
\sum_{i=1}^{3} D \varphi^{\circ} (\nu^{i}_{0}) =0 &  \text{ where we set  } \nu^{i}_{0}:= \frac{(\sigma^{i}_{x})^{\perp}}{|\sigma^{i}_{x}|},
\end{array}
\right.
\end{align}
as well as
\begin{align}\label{cc2}
&\frac{\sigma^{i}_{xx}}{|\sigma^{i}_{x}|^{2}} =0 \qquad  \text{ at }  x=1 \qquad \text{ for } i=1,2,3\\\label{cc3}
& \psi(\theta^{i}_{0})\frac{\sigma^{i}_{xx}}{|\sigma^{i}_{x}|^{2}}=
 \psi(\theta^{j}_{0})\frac{\sigma^{j}_{xx}}{|\sigma^{j}_{x}|^{2}} \qquad   \text{ at }  x=0 \qquad  \text{ for } i, j \in 
 \{ 1,2,3 \}
\end{align} 
where 
\begin{align*}
\psi(\theta^{i}_{0}):=\varphi^{\circ}(\nu^{i}_{0}) (D^2 \varphi^\circ(\nu^{i}_{0}) \frac{\sigma^{i}_{x}}{|\sigma^{i}_{x}|}\cdot \frac{\sigma^{i}_{x}}{|\sigma^{i}_{x}|}).  
\end{align*}

Existence and uniqueness in the isotropic case have been shown in Bronsard and Reitich~\cite{BR}.
There the short-time existence proof is carried out in three steps: first a linearization around the initial data is performed, second the classical theory for parabolic system is used to prove existence for the linearized system, third a fixed-point argument is applied to obtain  short-time existence for the original non-linear problem.
Due to the presence of the anisotropy map the problem is now clearly highly nonlinear 
and some details require attention.
In the following we provide the main arguments. With respect to \cite{BR}
one striking difference consists in the treatment of the boundary condition at the triple junction. In the isotropic case \eqref{HC} yields $\tau^{1}+\tau^{2}+\tau^{3}=0$, which gives an angle condition described in \cite[eq.(28)]{BR} as $\tau^{1} \cdot \tau^{2}= \cos (2\pi/3)=\tau^{2} \cdot \tau^{3}$. The latter two equations are then accordingly linearized around the initial datum.
Here we need to work with \eqref{HC} directly, since $\varphi^{\circ}$ is a given arbitrary (smooth and elliptic) anisotropy map.

\smallskip
\noindent\textbf{Function spaces and notation.}
For the convenience of the reader let us  recall the definition of the parabolic H\"older spaces (recall \cite[page 66 and 91]{Sol}) and fix some notation.

For a function $v : [0,T] \times [0,1] \to \R$ and $\rho \in (0,1)$ we let
\begin{align*}
[v]_{\rho,x}&:= \sup_{(t,x), (t,y) \in [0,T] \times [0,1]} \frac{|v(t,x)-v(t,y) |}{|x-y|^{\rho}},\\
[v]_{\rho,t}&:= \sup_{(t,x), (t',x) \in [0,T] \times [0,1]} \frac{|v(t,x)-v(t',x) |}{|t-t'|^{\rho}}.
\end{align*}
For $\alpha \in (0,1)$ and $k \in \mathbb{N}_0$ we define 
$C^{ \frac{k+\alpha}{2},k+\alpha} ([0,T] \times [0,1]) $ 
to be the space of all maps $v : [0,T] \times [0,1] \to \R$ with continuous derivatives $\partial_{t}^{i}\partial_{x}^{j}v$ for  $i, j \in \mathbb{N} \cup \{0 \}$ with $2i+j \leq k$ and such that the norm 
\begin{align*}
\|v \|_{ C^{ \frac{k+\alpha}{2}, k+\alpha}([0,T] \times [0,1]) } &:= \sum_{2i+j=0}^{k} \sup_{(t,x) \in [0,T] \times [0,1]} |\partial_{t}^{i}\partial_{x}^{j}v (t,x)| \\
& \quad+ \sum_{2i+j=k} [\partial_{t}^{i}\partial_{x}^{j}v ]_{\alpha,x} + \sum_{0<k+\alpha-2i-j <2}  [\partial_{t}^{i}\partial_{x}^{j}v ]_{\frac{k+\alpha -2i-j}{2},t}
\end{align*}
is finite. Note that $C^{ \frac{2+\alpha}{2},2+\alpha} ([0,T] \times [0,1]) \subset C^{ \frac{1+\alpha}{2},1+\alpha} ([0,T] \times [0,1]) \subset C^{ \frac{\alpha}{2},\alpha} ([0,T] \times [0,1])$.
We adopt the following conventions:
 \begin{itemize}
 \item in the proofs, and whenever clear from the context, we do not write the set of the parabolic H\"older spaces. In other words  we simply write  $\|v \|_{C^{\frac{k+\alpha}{2}, k+\alpha}}$ instead of $\|v \|_{ C^{ \frac{k+\alpha}{2}, k+\alpha}([0,T] \times [0,1]) }$; 
 \item for H\"older norms  on spaces in only one variable we always write the set, for instance in $C^{2,\alpha}([0,1])$ or $C^{0,\frac{\alpha}{2}}([0,T])$; 
\item the  $C^{\frac{k+\alpha}{2},k+\alpha }$-norm of a  \emph{vector-valued} map is the sum of the norms of its components. 
\end{itemize}
Useful lemmas for parabolic H\"older spaces are collected in Appendix~\ref{AppA}.

\subsection{Linearized Problem}
For some $ 0 <T<1$ and $M>0$ to be chosen later on (cf.~\eqref{choiceM})  define
\begin{align}\label{defXi}
X_{i}= \{ v \in C^{\frac{2+\alpha}{2}, 2+\alpha}([0,T] \times [0,1]; \R^{2})\, : \| v \|_{C^{\frac{2+\alpha}{2}, 2+\alpha}} \leq M, v(0, \cdot)=\sigma^{i}(\cdot) \} 
\end{align}
for $i=1,2,3$.
Furthermore let $\delta:= \min \{ |\sigma^{i}_{x}(x)|\,: \, x \in [0,1], \text{ and } i \in \{ 1,2,3 \} \}$. It is $\delta>0$.
Upon considering $\sigma^{i}$ as a map $\sigma^{i} \in C^{\frac{2+\alpha}{2}, 2+\alpha}([0,T] \times [0,1]; \R^{2})$ by extending it as a constant function in time, similar reasoning as in \cite[Lemma~3.1]{DLP-STE} (using now Lemma~\ref{B5}) yields that it is possible to choose $T=T(M, \delta, \sigma)$ so small in the definition of $X_{i}$ above so that any map $v \in X_{i}$ is regular for all times. 
From now on we assume that $T$ is fixed in such a way that the regularity of the curves is guaranteed, that is
\begin{align}\label{eq:immersion}
|\bar{u}_{x}^{i}(t,x)| \geq \frac12\delta \text{ for all } (t,x) \in [0,T] \times [0,1]
\end{align}
for any $\bar{u}^{i} \in X_{i}$, $i=1,2,3$.
As in \cite{BR} we seek a fixed point of the map
\begin{align}\label{defiR}
\mathcal{R}: \prod_{j=1}^{3} X_{j} &\to  \prod_{j=1}^{3} X_{j} 
\\
\bar{u}=(\bar{u}^{1}, \bar{u}^{2}, \bar{u}^{3}) &\mapsto \mathcal{R}\bar{u}= u= (u^{1}, u^{2}, u^{3}) \notag
\end{align}
where  $u$ solves the  following linearized system, which we refer to as the linear problem.

\smallskip

\noindent \textbf{The Linear Problem (LP)}
Given $\bar{u}=(\bar{u}^{1}, \bar{u}^{2}, \bar{u}^{3}) \in \prod_{j=1}^{3} X_{j}$
we look for  $u =(u^{1}, u^{2}, u^{3})$, $u \in \prod_{j=1}^{3}  C^{\frac{2+\alpha}{2}, 2+\alpha}([0,T] \times [0,1]; \R^{2})$ solution to
\begin{align}
u_{ t}^{j} -D_{j}u_{xx}^{j} &= f^{j}\\
u^{j}(0, x)=\sigma^{j}(x)
\end{align}
 where
\begin{align*}
D_{j}= \frac{\psi(\theta^{j}_{0})}{|\sigma_{x}^{j}|^{2}}>0, \qquad  f^{j}:= \left(\frac{\psi(\bar{\theta}^{j})}{|\bar{u}_{x}^{j}|^{2}} - \frac{\psi(\theta^{j}_{0})}{|\sigma_{x}^{j}|^{2}} \right) \bar{u}_{xx}^{j} \in \R^{2}
\end{align*}
for $j=1,2,3$, with (the linearized) boundary conditions (recall \eqref{OP2} and $\nu^{i}_{0}= \frac{(\sigma^{i}_{x})^{\perp}}{|\sigma^{i}_{x}|}$)
\begin{align*}
&u^{i}(t,1)=P^{i}  \qquad \forall t \in (0,T), \quad i=1,2,3 \\
&u^{1}(t,0)= u^{2}(t,0) =u^{3}(t,0)   \qquad \forall t \in (0,T)\\
&\sum_{i=1}^{3} \left( \varphi^{\circ}(\nu^{i}_{0}) \frac{(u_{x}^{i})^{\perp}}{|\sigma^{i}_{x}|} + (D\varphi^{\circ}(\nu^{i}_{0}) \cdot \frac{\sigma^{i}_{x}}{|\sigma_{x}^{i}|}) \frac{u^{i}_{x}}{|\sigma_{x}^{i}|} \right) =
\sum_{i=1}^{3} \left( \varphi^{\circ}(\nu^{i}_{0}) \frac{1}{|\sigma^{i}_{x}|} -  \varphi^{\circ}(\bar{\nu}^{i}) \frac{1}{|\bar{u}^{i}_{x}|}  \right)(\bar{u}_{x}^{i})^{\perp} \\
& \qquad \qquad + \left( (D\varphi^{\circ}(\nu^{i}_{0}) \cdot \frac{\sigma^{i}_{x}}{|\sigma_{x}^{i}|})\frac{1}{|\sigma_{x}^{i}|}  -
(D\varphi^{\circ}(\bar{\nu}^{i}) \cdot \bar{\tau}^{i}) \frac{1}{|\bar{u}_{x}^{i}|}\right)\bar{u}_{x}^{i}=:\bar{b} \qquad \forall t \in (0,T).
\end{align*}

\smallskip
\noindent \textbf{Solution of the linear problem (LP)}

As in \cite{BR} we follow the theory developed in \cite{Sol}. The above system can be written as $\mathcal{L}(x,t,\partial_{x}, \partial_{t})u$, with $\mathcal{L}(x,t,\partial_{x}, \partial_{t})=diag (l_{kk})_{k=1}^{6}$ where
\begin{align*}
l_{kk}(x,t,\partial_{x}, \partial_{t}) = \partial_{t} - D_{i} \partial_{x}^{2} \qquad \text{ if } k=2(i-1)+j
\end{align*}
 for some $j \in \{1,2 \}$ and $i \in \{ 1,2,3 \}$. In the following let  (for $\mathrm{i}=\sqrt{-1}$, $\xi\in \R$, $p \in \mathbb{C}$ )
 \begin{align*}
 L(x,t,\mathrm{i} \xi, p)&:= \det \mathcal{L}(x,t,\mathrm{i} \xi, p)=\prod_{i=1}^{3}(p+D_{i}\xi^{2})^{2}\\
 \hat{\mathcal{L}}(x,t,\mathrm{i} \xi, p) &:=L(x,t,\mathrm{i} \xi, p) \mathcal{L}^{-1}(x,t,\mathrm{i} \xi, p)=diag(A_{kk})_{k=1}^{6}
 \end{align*}
 with
 \begin{align*}
 A_{kk}=A_{kk}(x,t,\mathrm{i} \xi, p)=\frac{\prod_{i=1}^{3}(p+D_{i}\xi^{2})^{2}}{p+ D_{l}\xi^{2}} \qquad \text{if } k=2(l-1)+j
 \end{align*}
for $l\in \{1,2 ,3\}$ and $j \in \{1,2\}$. Since many terms coincide in the following we simply write
\begin{align*}
A_{1}:=A_{11}=A_{22}, \qquad A_{2}:=A_{33}=A_{44}, \qquad A_{3}:=A_{55}=A_{66}.
\end{align*}
As in \cite{BR} we note that the \emph{parabolicity condition} \cite[p.~8]{Sol} is fulfilled since for any $i=1,2,3$ we have that $$D_{i} \geq m \cdot \min \left\{ \frac{1}{|\sigma_{x}^{j}(x)|^{2}}\, :\, j=1,2,3, \, x \in [0,1] \right\} >0 $$
where $m$ is as in \eqref{boundmpiccolo}.

At the boundary we need to check the so-called \emph{complementary conditions} \cite[p.~11]{Sol}. First of all we consider the system of boundary conditions at the junction point at $x=0$.
Here the system reads $\mathcal{B}u=\left(\begin{array}{c}0\\ 0\\ \bar{b}  \end{array}\right)$ where $u=(u^{1}, u^{2}, u^{3}) \in \R^{6}$ with $\mathcal{B}$ a $6\times6$ matrix given by

\begin{align*}
\mathcal{B}(x=0,t, \partial_{x}, \partial_{t})=
\left(\begin{array}{ccc}
Id & -Id & 0 \\
0 & Id & -Id \\
Q_{1} & Q_{2} &Q_{3}
\end{array} \right)
\end{align*}
where each block entry is a $(2\times2)$ matrix with
\begin{align*}
Q_{i}:=\frac{\varphi^{\circ}(\nu^{i}_{0})}{|\sigma^{i}_{x}|} 
\left(\begin{array}{cc}
0& -\partial_{x}\\
\partial_{x}& 0
\end{array} \right)
 + (D\varphi^{\circ}(\nu^{i}_{0}) \cdot \frac{\sigma^{i}_{x}}{|\sigma_{x}^{i}|^{2}} ) \left(\begin{array}{cc}
\partial_{x}& 0\\
0& \partial_{x}
\end{array} \right) 
\end{align*}
with all coefficients evaluated at $x=0$.
Therefore we obtain
\begin{align*}
\mathcal{B}(x=0,t, \mathrm{i}\tau, p)= \left(
\begin{array}{cccccc}
1&0&-1& 0&0&0\\0&1&0&-1&0&0\\
0&0&1&0&-1&0 \\ 0&0&0&1&0&-1\\
\mathrm{i}\tau b_{51} & -\mathrm{i}\tau b_{52} &  \mathrm{i}\tau b_{53} & -\mathrm{i}\tau b_{54}& \mathrm{i}\tau b_{55} & -\mathrm{i}\tau b_{56}\\
\mathrm{i}\tau b_{52} & \mathrm{i}\tau b_{51} &\mathrm{i}\tau b_{54} & \mathrm{i}\tau b_{53}&\mathrm{i}\tau b_{56} & \mathrm{i}\tau b_{55}
\end{array}
\right)
\end{align*}
where
\begin{align*}
b_{51}&=(D\varphi^{\circ}(\nu^{1}_{0}) \cdot \frac{\sigma^{1}_{x}}{|\sigma_{x}^{1}|^{2}} ),  & b_{52}=\frac{\varphi^{\circ}(\nu^{1}_{0})}{|\sigma^{1}_{x}|} \\
b_{53}&=(D\varphi^{\circ}(\nu^{2}_{0}) \cdot \frac{\sigma^{2}_{x}}{|\sigma_{x}^{2}|^{2}} ),  & b_{54}=\frac{\varphi^{\circ}(\nu^{2}_{0})}{|\sigma^{2}_{x}|}\\
b_{55}&=(D\varphi^{\circ}(\nu^{3}_{0}) \cdot \frac{\sigma^{3}_{x}}{|\sigma_{x}^{3}|^{2}} ),  & b_{56}=\frac{\varphi^{\circ}(\nu^{3}_{0})}{|\sigma^{3}_{x}|}
\end{align*}
with all expressions evaluated at $x=0$. In the isotropic case $b_{51}=b_{53}=b_{55}=0$ and 
$\frac{\varphi^{\circ}(\nu^{i}_{0})}{|\sigma^{i}_{x}|} = \frac{1}{|\sigma^{i}_{x}|}$.
Next note that as a function of $\tau$ the polynomial $L(x,t,\mathrm{i} \tau, p)$ has six roots with positive imaginary parts and six roots with negative imaginary parts provided $Re(p) \geq 0$ and $p \neq0$. More precisely writing $p= |p| e^{\mathrm{i} \theta_{p}}$ with $-\pi/2 \leq \theta_{p} \leq \pi/2$ and $|p|\neq0$ we may write
\begin{align*}
L(x,t,\mathrm{i} \tau, p)=\prod_{i=1}^{3}D_{i}^{2}(\tau-\tau_{i}^{+})^{2}(\tau-\tau_{i}^{-})^{2}
\end{align*}
with 
\begin{align*}
\tau_{i}^{+ }=\tau_{i}^{+ }(x,p) = \sqrt{\frac{|p|}{D_{i}}} e^{\mathrm{i}\left(\frac{\pi}{2}+ \frac{\theta_{p}}{2} \right)}=\mathrm{i} \sqrt{\frac{p}{D_{i}}}\qquad \qquad 
\tau_{i}^{- }=\tau_{i}^{- }(x,p)= \sqrt{\frac{|p|}{D_{i}}} e^{\mathrm{i} \left(\frac{3\pi}{2}+ \frac{\theta_{p}}{2} \right)} =-\mathrm{i} \sqrt{\frac{p}{D_{i}}}.
\end{align*}
Following \cite[p.~11]{Sol} we set
\begin{align*}
M^{+}=M^{+}(x,\tau, p)= \prod_{i=1}^{3}(\tau-\tau_{i}^{+})^{2}.
\end{align*}
By \cite[p.~11]{Sol} the complementary condition at $x=0$ is satisfied  if the rows of the matrix
\begin{align*}
\mathcal{A}(x=0,t,\mathrm{i} \tau, p):=\mathcal{B}(x=0,t, \mathrm{i}\tau, p)\hat{\mathcal{L}}(x=0,t, \mathrm{i}\tau, p)
\end{align*}
are linearly independent modulo $M^{+}$ whereby $p\neq 0$, $Re (p) \geq0$. Therefore we need to verify that if there exists $w \in 
\R^{6}$ such that
\begin{align*}
w^{T} \cdot \mathcal{A}(x=0,t,\mathrm{i} \tau, p) = (0,0,0,0,0,0) \mod M^{+}
\end{align*}
then $w=\vec{0}$.
This gives the six equations
\begin{align*}
&A_{1}(w_{1}+ w_{5}\mathrm{i} \tau b_{51} + w_{6}\mathrm{i} \tau b_{52}) =0 \mod M^{+}\\
& A_{1}(w_{2} -w_{5}\mathrm{i} \tau b_{52}+  w_{6}\mathrm{i} \tau b_{51} )=0 \mod M^{+}\\
&A_{2} (-w_{1}+w_{3} +w_{5}\mathrm{i} \tau b_{53}+ w_{6}\mathrm{i} \tau b_{54})=0 \mod M^{+}\\
&A_{2} (-w_{2}+w_{4} -w_{5}\mathrm{i} \tau b_{54}+  w_{6}\mathrm{i} \tau b_{53} )=0 \mod M^{+}\\
&A_{3} (-w_{3}+ w_{5}\mathrm{i} \tau b_{55} + w_{6}\mathrm{i} \tau b_{56})=0 \mod M^{+}\\
&A_{3} (-w_{4}-w_{5}\mathrm{i} \tau b_{56}+  w_{6}\mathrm{i} \tau b_{55} )=0 \mod M^{+}.
\end{align*}
Using the fact that $A_{i}$ and $M^{+}$ have many factors in common, we infer that the first equation in equivalent to
\begin{align*}
p_{1}(\tau)(w_{1}+ w_{5}\mathrm{i} \tau b_{51} + w_{6}\mathrm{i} \tau b_{52})=0 \mod (\tau- \tau_{1}^{+})
\end{align*}
where 
\begin{align*}
p_{1}(\tau)=(\tau-\tau_{1}^{-})(\tau-\tau_{2}^{-})^{2}(\tau-\tau_{3}^{-})^{2}.
\end{align*}
Since $(\tau- \tau_{1}^{+})$ can not divide $p_{1}(\tau)$ then $\tau_{1}^{+}$ must be a root of the remaning linear factor.
Reasoning in a similar way for the other five equations we obtain that $w$ must satisfy the system
\begin{align*}
w_{1} + w_{5}\mathrm{i}  b_{51}\tau_{1}^{+} + w_{6}\mathrm{i}  b_{52}\tau_{1}^{+}&=0\\
w_{2} -w_{5}\mathrm{i} \tau_{1}^{+} b_{52}+  w_{6}\mathrm{i} \tau_{1}^{+} b_{51} &=0\\
-w_{1}+w_{3} +w_{5}\mathrm{i} \tau_{2}^{+} b_{53}+ w_{6}\mathrm{i} \tau_{2}^{+} b_{54}&=0\\
-w_{2}+w_{4} -w_{5}\mathrm{i} \tau_{2}^{+} b_{54}+  w_{6}\mathrm{i} \tau_{2}^{+} b_{53} &=0\\
-w_{3}+ w_{5}\mathrm{i} \tau_{3}^{+} b_{55} + w_{6}\mathrm{i} \tau_{3}^{+} b_{56}&=0\\
-w_{4}-w_{5}\mathrm{i} \tau_{3}^{+} b_{56}+  w_{6}\mathrm{i} \tau_{3}^{+} b_{55}&=0
\end{align*}
for whose determinant we compute
\begin{align*}
& \det \left(\begin{array}{cccccc}
1&0&0&0&\mathrm{i}  b_{51}\tau_{1}^{+} &\mathrm{i}  b_{52}\tau_{1}^{+}\\
0&1&0&0&-\mathrm{i} \tau_{1}^{+} b_{52} & \mathrm{i} \tau_{1}^{+} b_{51}\\
-1&0&1&0&\mathrm{i} \tau_{2}^{+} b_{53} & \mathrm{i} \tau_{2}^{+} b_{54}\\
0&-1&0&1&-\mathrm{i} \tau_{2}^{+} b_{54} &\mathrm{i} \tau_{2}^{+} b_{53}\\
0&0&-1&0&\mathrm{i} \tau_{3}^{+} b_{55}& \mathrm{i} \tau_{3}^{+} b_{56}\\
0&0&0&-1&-\mathrm{i} \tau_{3}^{+} b_{56} &\mathrm{i} \tau_{3}^{+} b_{55}
\end{array} \right) \\
 &= - (b_{52} \tau_{1}^{+}+ b_{54} \tau_{2}^{+} + b_{56}\tau_{3}^{+})^{2}  -(b_{51} \tau_{1}^{+}+ b_{53} \tau_{2}^{+} + b_{55}\tau_{3}^{+})^{2}\\
 & = - \left(\sum_{i=1}^{3} \mathrm{i} 
  \frac{\varphi^{\circ}(\nu^{i}_{0})}{|\sigma_{x}^{i}|}\sqrt{p\frac{|\sigma_{x}^{i}|^{2}}{\psi(\theta^{i}_{0})}}\right)^{2} 
  - \left(\sum_{i=1}^{3} \mathrm{i} 
  \frac{D\varphi^{\circ}(\nu^{i}_{0}) \cdot \sigma_{x}^{i}}{|\sigma_{x}^{i}|^{2}}\sqrt{p\frac{|\sigma_{x}^{i}|^{2}}{\psi(\theta^{i}_{0})}}\right)^{2}\\
  & = p\left(\sum_{i=1}^{3}  
  \frac{\varphi^{\circ}(\nu^{i}_{0})}{\sqrt{\psi(\theta^{i}_{0})}}\right)^{2} + p\left(\sum_{i=1}^{3}  
  \frac{D\varphi^{\circ}(\nu^{i}_{0}) \cdot \sigma_{x}^{i}}{|\sigma_{x}^{i}|}\sqrt{\frac{1}{\psi(\theta^{i}_{0})}}\right)^{2}
\neq 0
 \end{align*}
 since $p \neq 0$ and $\left(\sum_{i=1}^{3}  
  \frac{\varphi^{\circ}(\nu^{i}_{0})}{\sqrt{\psi(\theta^{i}_{0})}}\right)^{2} >0$. It follows that $w=\vec{0}$ and the complementary condition at $x=0$ is fulfilled. Checking the complementary condition at $x=1$ is done in a similar way, but here computations are much simplier since $\mathcal{B}(x=1,t, \mathrm{i}\tau, p)$ is  given by the identity matrix.
  
 Finally  we observe that at $t=0$ the initial condition is given by the system $\mathcal{C}u=\sigma$ where $\mathcal{C}\in \R^{6\times 6}$ is  the identity matrix. The \emph{complementary condition}  here (cf. \cite[p~12]{Sol}) requires that the rows of the matrix
 $\mathcal{D}(x,p)=\mathcal{C}\cdot \hat{\mathcal{L}}(x,0,0,p)$ are linearly independent modulo $p^{6}$ at each point $x \in (0,1)$.
 This is readily checked.
 
Using \eqref{cc1}, \eqref{cc2}, \eqref{cc3} and the definition of the spaces $X_{j}$ we also observe that the linear problem fulfills the compatibility conditions of order zero (cf. \cite[p.~98]{Sol}). Application of \cite[Thm.~4.9]{Sol}
 yields the existence of a unique solution $u \in \prod_{j=1}^{3}  C^{\frac{2+\alpha}{2}, 2+\alpha}([0,T] \times [0,1]; \R^{2})$ satisying
\begin{align}\label{boundU}
&\sum_{i=1}^{3} \|  u^{i} \|_{C^{\frac{2+\alpha}{2}, 2+\alpha}([0,T] \times [0,1])} \\
& \leq C_{0} \left( 
\sum_{i=1}^{3} ( \| f^{i}\|_{C^{\frac{\alpha}{2},\alpha}([0,T] \times [0,1])} + \| \sigma^{i}\|_{C^{2,\alpha}([0,1])} +|P^{i}| )
+ \| \bar{b} \|_{C^{0, \frac{1+\alpha}{2}}([0,T])}
\right). \notag
\end{align}
 
 \subsection{Fixed point argument}
 
 Let $u=(u^{1}, u^{2}, u^{3}) \in \prod_{j=1}^{3}  C^{\frac{2+\alpha}{2}, 2+\alpha}([0,T] \times [0,1]; \R^{2})$ be the solution of the linear problem \textbf{(LP)}. We would like to verify the self-map and self-contraction property of the operator $\mathcal{R}$ (recall \eqref{defiR}).
 To that end we employ  \eqref{boundU}.
 
 \smallskip
\noindent \textbf{Self-map property}
 We need to estimate the right-hand side in \eqref{boundU}. For $j=1,2,3$ and using the definition of $X_{j}$ as well as Lemma~\ref{B2} we compute 
 \begin{align*}
  \| f^{j}\|_{C^{\frac{\alpha}{2},\alpha}} &\leq C \left \|\left(\frac{\psi(\bar{\theta}^{j})}{|\bar{u}_{x}^{j}|^{2}} - \frac{\psi(\theta^{j}_{0})}{|\sigma_{x}^{j}|^{2}} \right) \right\|_{C^{\frac{\alpha}{2},\alpha}} \|  \bar{u}_{xx}^{j}\|_{C^{\frac{\alpha}{2},\alpha}}\\
  &\leq C M  \left(\| \psi(\bar{\theta}^{j}) -\psi(\theta^{j}_{0}) \|_{C^{\frac{\alpha}{2},\alpha}} \left\| \frac{1}{|\bar{u}_{x}^{j}|^{2}}\right \|_{C^{\frac{\alpha}{2},\alpha}} + \|\psi(\theta^{j}_{0}) \|_{C^{\frac{\alpha}{2},\alpha}}
 \left \|  \frac{1}{|\bar{u}_{x}^{j}|^{2}} - \frac{1}{|\sigma_{x}^{j}|^{2}}  \right \|_{C^{\frac{\alpha}{2},\alpha}} \right).
 \end{align*}
Writing out the expressions of type $\psi(\theta)$ in terms of tangents and normals (recall \eqref{defpsi}, \eqref{a1}), manipulating them appropriately into products of differences (similarly to what we have done above) and application of Remark~\ref{B1} and of Lemmas \ref{B2}, \ref{B3}, \ref{lem:hilfsatz}, \ref{hilfsatz2}, and \ref{hilfsatz3} yields
 \begin{align*}
 \| f^{j}\|_{C^{\frac{\alpha}{2},\alpha}} &\leq C_{1} T^{\frac{\alpha}{2}}
 \end{align*}
  where $C_{1}=C_{1}(\delta, \|\sigma^{j}\|_{C^{2,\alpha}([0,1])}, M, \|\varphi^{\circ}\|_{C^{4}})$.
 Next we write
 \begin{align*}
 \bar{b}&= \sum_{i=1}^{3} \left( [\varphi^{\circ}(\nu^{i}_{0})- \varphi^{\circ}(\bar{\nu}^{i}) ]\frac{1}{|\sigma^{i}_{x}|} +  \varphi^{\circ}(\bar{\nu}^{i}) [\frac{1}{|\sigma^{i}_{x}|} - \frac{1}{|\bar{u}^{i}_{x}|} ] \right)(\bar{u}_{x}^{i})^{\perp} \\
&  + \left( [D\varphi^{\circ}(\nu^{i}_{0}) \cdot \frac{\sigma^{i}_{x}}{|\sigma_{x}^{i}|} -D\varphi^{\circ}(\bar{\nu}^{i}) \cdot \bar{\tau}_{i}]\frac{1}{|\sigma_{x}^{i}|}  +
(D\varphi^{\circ}(\bar{\nu}^{i}) \cdot \bar{\tau}_{i}) [\frac{1}{|\sigma_{x}^{i}|}- \frac{1}{|\bar{u}_{x}^{i}|}]\right)\bar{u}_{x}^{i}
 \end{align*} 
Similar considerations yield now
\begin{align*}
\|\bar{b} \|_{C^{0, \frac{1+\alpha}{2}}([0,T])} \leq C_{2}T^{\frac{\alpha}{2}}
\end{align*}
 with $C_{1}=C_{1}(\delta, \|\sigma^{j}\|_{C^{2,\alpha}([0,1])}, M, \|\varphi^{\circ}\|_{C^{3}})$. Putting all estimates together we derive from~\eqref{boundU}
\begin{align*}
 \|  u^{i} \|_{C^{\frac{2+\alpha}{2}, 2+\alpha}([0,T] \times [0,1])} 
& \leq 3 C_{0} (C_{1}+C_{2})T^{\frac{\alpha}{2}} + C_{0}  
\sum_{i=1}^{3} ( \| \sigma^{i}\|_{C^{2,\alpha}([0,1])} +|P^{i}| ).
 \end{align*}
 Hence choosing
 \begin{align}\label{choiceM}
 M:=2C_{0} 
\sum_{i=1}^{3} ( \| \sigma^{i}\|_{C^{2,\alpha}([0,1])} +|P^{i}| )
 \end{align}
and taking $T<1$ so that $3 C_{0} (C_{1}+C_{2})T^{\frac{\alpha}{2}} \leq M/2$ we infer  that $\mathcal{R}$ maps $X_{1} \times X_{2} \times X_{3}$ into itself. This will be assumed henceforth.

\smallskip
\noindent \textbf{Contraction property}
Let $u=(u^{1}, u^{2}, u^{3})=\mathcal{R}(\bar{u})$ and $v=(v^{1}, v^{2}, v^{3})=\mathcal{R}(\bar{v}) \in \prod_{j=1}^{3} X_{j}$ be two solutions of the linear problem \textbf{(LP)}. Set $w=(w^{1},w^{2},w^{3})$ with $w^{j}=u^{j}-v^{j}$, $j=1,2,3$. Then the $w^{j}$'s satisfy
\begin{align}
w_{ t}^{j} -D_{j}w_{xx}^{j} &= f^{j}(\bar{u})-f^{j}(\bar{v})\\
w^{j}(0, x)=0
\end{align}
 where $D_{j}= \frac{\psi(\theta^{j}_{0})}{|\sigma_{x}^{j}|^{2}}>0$ and 
\begin{align*}
  f^{j}(\bar{u})-f^{j}(\bar{v}):= \left(\frac{\psi(\theta(\bar{u})^{j})}{|\bar{u}_{x}^{j}|^{2}} - \frac{\psi(\theta^{j}_{0})}{|\sigma_{x}^{j}|^{2}} \right) \bar{u}_{xx}^{j} -\left(\frac{\psi(\theta(\bar{v})^{j})}{|\bar{v}_{x}^{j}|^{2}} - \frac{\psi(\theta^{j}_{0})}{|\sigma_{x}^{j}|^{2}} \right) \bar{v}_{xx}^{j}
\end{align*}
for $j=1,2,3$, with  boundary conditions (recall that $\nu^{i}_{0}= \frac{(\sigma^{i}_{x})^{\perp}}{|\sigma^{i}_{x}|}$ and note that  here $\psi(\theta(u))$ is given by \eqref{a1} and \eqref{defpsi} with tangent and normal vector of the curve $u$)
\begin{align*}
&w^{i}(t,1)=0  \qquad \forall t \in [0,T], \quad i=1,2,3 \\
&w^{1}(t,0)= w^{2}(t,0) =w^{3}(t,0)   \qquad \forall t \in [0,T]\\
&\sum_{i=1}^{3} \left( \varphi^{\circ}(\nu^{i}_{0}) \frac{(w_{x}^{i})^{\perp}}{|\sigma^{i}_{x}|} + (D\varphi^{\circ}(\nu^{i}_{0}) \cdot \frac{\sigma^{i}_{x}}{|\sigma_{x}^{i}|}) \frac{w^{i}_{x}}{|\sigma_{x}^{i}|} \right) =
\sum_{i=1}^{3} \left( \varphi^{\circ}(\nu^{i}_{0}) \frac{1}{|\sigma^{i}_{x}|} -  \varphi^{\circ}(\nu(\bar{u})^{i}) \frac{1}{|\bar{u}^{i}_{x}|}  \right)(\bar{u}_{x}^{i})^{\perp} \\
& \qquad \qquad + \left( (D\varphi^{\circ}(\nu^{i}_{0}) \cdot \frac{\sigma^{i}_{x}}{|\sigma_{x}^{i}|})\frac{1}{|\sigma_{x}^{i}|}  -
(D\varphi^{\circ}(\nu(\bar{u})^{i}) \cdot \tau(\bar{u})^{i}) \frac{1}{|\bar{u}_{x}^{i}|}\right)\bar{u}_{x}^{i} \\
&\qquad\qquad -\sum_{i=1}^{3} \left( \varphi^{\circ}(\nu^{i}_{0}) \frac{1}{|\sigma^{i}_{x}|} -  \varphi^{\circ}(\nu(\bar{v})^{i}) \frac{1}{|\bar{v}^{i}_{x}|}  \right)(\bar{v}_{x}^{i})^{\perp} \\
& \qquad \qquad - \left( (D\varphi^{\circ}(\nu^{i}_{0}) \cdot \frac{\sigma^{i}_{x}}{|\sigma_{x}^{i}|})\frac{1}{|\sigma_{x}^{i}|}  -
(D\varphi^{\circ}(\nu(\bar{v})^{i}) \cdot \tau(\bar{v})^{i}) \frac{1}{|\bar{v}_{x}^{i}|}\right)\bar{v}_{x}^{i}  =:b(\bar{u})-b(\bar{v})
\qquad \forall t \in [0,T].
\end{align*}
This is again a linear parabolic system and it satisfies the complementary and compatibility conditions. In particular it satisfies the Schauder-type estimate 
\begin{align}\label{boundW}
&\sum_{i=1}^{3} \|  w^{i} \|_{C^{\frac{2+\alpha}{2}, 2+\alpha}([0,T] \times [0,1])} \\
& \leq C_{0} \left( 
\sum_{i=1}^{3} ( \| f^{i}(\bar{u})-f^{i}(\bar{v})\|_{C^{\frac{\alpha}{2},\alpha}([0,T] \times [0,1])} )
+ \| b(\bar{u})-b(\bar{v}) \|_{C^{0, \frac{1+\alpha}{2}}([0,T])}
\right). \notag
\end{align}
Using the lemmas from the Appendix~\ref{AppA}, the definition of $X_{j}$, and arguments similar to those employed in the verification of the self-map property we  compute for $j=1,2,3$
\begin{align*}
\| f^{j}(\bar{u})-f^{j}(\bar{v})\|_{C^{\frac{\alpha}{2},\alpha}}& \leq 
\left\| \frac{\psi(\theta(\bar{u})^{j})}{|\bar{u}_{x}^{j}|^{2}} - \frac{\psi(\theta(\bar{v})^{j})}{|\bar{v}_{x}^{j}|^{2}}  \right\|_{C^{\frac{\alpha}{2},\alpha}} \|\bar{u}_{xx}^{j}\|_{C^{\frac{\alpha}{2},\alpha}} \\
& \quad + \left \| \frac{\psi(\theta(\bar{v})^{j})}{|\bar{v}_{x}^{j}|^{2}} -\frac{\psi(\theta^{j}_{0})}{|\sigma_{x}^{j}|^{2}}  \right\|_{C^{\frac{\alpha}{2},\alpha}} \|\bar{u}_{xx}^{j} - \bar{v}_{xx}^{j}\|_{C^{\frac{\alpha}{2},\alpha}}\\
& \leq C T^{\frac{\alpha}{2}}\| \bar{u}^{j} -\bar{v}^{j} \|_{C^{\frac{2+\alpha}{2},2+\alpha}}
\end{align*}
and
\begin{align*}
\| b(\bar{u})&-b(\bar{v}) \|_{C^{0, \frac{1+\alpha}{2}}([0,T])} \\ & \leq
\sum_{i=1}^{3} \Big(
\left \|  \varphi^{\circ}(\nu^{i}_{0}) \frac{1}{|\sigma^{i}_{x}|} -  \varphi^{\circ}(\nu(\bar{u})^{i}) \frac{1}{|\bar{u}^{i}_{x}|}    \right \|_{C^{0, \frac{1+\alpha}{2}}([0,T])} \|(\bar{u}_{x}^{i})^{\perp}  - (\bar{v}_{x}^{i})^{\perp} \|_{C^{0, \frac{1+\alpha}{2}}([0,T])} \\
& +\left\| \varphi^{\circ}(\nu(\bar{u})^{i}) \frac{1}{|\bar{u}^{i}_{x}|}  -  \varphi^{\circ}(\nu(\bar{v})^{i}) \frac{1}{|\bar{v}^{i}_{x}|}  \right  \|_{C^{0, \frac{1+\alpha}{2}}([0,T])} \|(\bar{v}_{x}^{i})^{\perp} \|_{C^{0, \frac{1+\alpha}{2}}([0,T])}  \\
&  + \left\| (D\varphi^{\circ}(\nu^{i}_{0}) \cdot \frac{\sigma^{i}_{x}}{|\sigma_{x}^{i}|})\frac{1}{|\sigma_{x}^{i}|}  -
(D\varphi^{\circ}(\nu(\bar{u})^{i}) \cdot \tau(\bar{u})^{i}) \frac{1}{|\bar{u}_{x}^{i}|}\right \|_{C^{0, \frac{1+\alpha}{2}}([0,T])}  \|\bar{u}_{x}^{i} - \bar{v}_{x}^{i}  \|_{C^{0, \frac{1+\alpha}{2}}([0,T])}  \\
&  + \left \|   (D\varphi^{\circ}(\nu(\bar{u})^{i}) \cdot \tau(\bar{u})^{i}) \frac{1}{|\bar{u}_{x}^{i}|}-
(D\varphi^{\circ}(\nu(\bar{v})^{i}) \cdot \tau(\bar{v})^{i}) \frac{1}{|\bar{v}_{x}^{i}|}
\right \|_{C^{0, \frac{1+\alpha}{2}}([0,T])}   
\| \bar{v}_{x}^{i} \|_{C^{0, \frac{1+\alpha}{2}}([0,T])}  \Big )\\
&\leq C T^{\frac{\alpha}{2}} \sum_{i=1}^{3}\| \bar{u}^{j} -\bar{v}^{j} \|_{C^{\frac{2+\alpha}{2},2+\alpha}}
\end{align*}
where $C=C(M, \delta,\|\sigma\|_{C^{2,\alpha}([0,1])}, \| \varphi^{\circ}\|_{C^{4}})$. 
Thus, by choosing $T$ possibly even smaller,  we obtain
$$\sum_{i=1}^{3} \|  w^{i} \|_{C^{\frac{2+\alpha}{2}, 2+\alpha}([0,T] \times [0,1])} \leq \frac{1}{2} \sum_{i=1}^{3} \|  \bar{w}^{i} \|_{C^{\frac{2+\alpha}{2}, 2+\alpha}([0,T] \times [0,1])} $$
and the contraction property of $\mathcal{R}$ is established.

\smallskip
Finally application of the Banach's fixed point theorem yields  the existence of a unique map $u \in\prod_{j=1}^{3} X_{j}$ with $u =\mathcal{R}(u)$, that is a solution to \eqref{OP1}, \eqref{OP2}.  In particular we can state the following theorem.

\begin{theo}\label{teo:STE}
Let $P^{i} \in \R^{2}$, $i=1,2,3$, be given points and $\alpha \in (0,1)$. Let $\sigma^{i} \in C^{2,\alpha}([0,1], \R^{2})$, $i=1,2,3$ be regular maps fulfilling the compatibility conditions \eqref{cc1}, \eqref{cc2}, \eqref{cc3}. Then there exists $T>0$ and unique regular maps $u^{i} \in C^{\frac{2+\alpha}{2}, 2+\alpha}([0,T] \times [0,1], \R^{2})$, $i=1,2,3$ such that \eqref{OP1}, \eqref{OP2} are satisfied together with the initial conditions $u^{i}(0, x)=\sigma^{i}(x)$, $x \in [0,1]$, $i=1,2,3$.
\end{theo}

\begin{cor}\label{cor:STE}
Let $u^{i} \in C^{\frac{2+\alpha}{2}, 2+\alpha}([0,T] \times [0,1], \R^{2})$, $i=1,2,3$ be the solutions found in Theorem~\ref{teo:STE}.
Then $u^{i} \in C^{\infty}((0,T] \times [0,1], \R^{2})$.
\end{cor}
\begin{proof}
The instant parabolic smoothing can be shown by some standard arguments employing a cut-off function and a boot-strap argument in the same fashion as in \cite[Thm~2.3]{DLP-STE}.
\end{proof}

\section{Maximal solution for the geometric problem}\label{sec:maxsolGEOP}

We now prove existence, uniqueness and regularity of a maximal geometric solution.
We first show that a geometric solution is also a solution to the special flow up to a diffeomorphism.

\begin{lemma}\label{lemma:diffeo}
Let $(u^{1}, u^{2}, u^{3})$, with $u^{i} \in C^{\frac{2+\alpha}{2}, 2+\alpha} ([0,T] \times [0,1], \R^{2}) $, $i=1,2,3$,  be a solution of the geometric problem (according to  Definition~\ref{geomsol}) with tangential components $\lambda^{i}= u_{t}^{i} \cdot \tau^{i}$. Then there exists a orientation preserving diffeomorphism $\phi^{i} \in C^{\frac{2+\alpha}{2}, 2+\alpha} ([0,T'] \times [0,1], [0,1]) $, $i=1,2,3$,  for some $0<T'\leq T$, such that $(\tilde{u}^{1}, \tilde{u}^{2}, \tilde{u}^{3})$, with $\tilde{u}^{i}(t,y):= u^{i}(t, \phi^{i}(t,y)) \in C^{\frac{2+\alpha}{2}, 2+\alpha} ([0,T'] \times [0,1], \R^{2}) $ is the solution of the Special Flow (recall Definition~\ref{def:SF} and Section~\ref{sec:2}).
\end{lemma}
\begin{proof} 
Since the proof of existence of $\phi^{i}$ is performed identically for every map $i=1,2,3$, 
 let us omit the index $i$ for simplicity of notation. 
Note that by the assumptions on the initial data (recall Definition~\ref{admtriod}) we have that the  anisotropic curvature (and hence the curvature and  curvature vector) vanishes at $x=1$ at time zero,  that is
\begin{align}\label{primacond}
(\kappa\nu)|_{(t=0,x=1)}=0.
\end{align}
Moreover we have that at the junction point at time zero there holds
\begin{align}\label{secondacond}
(\psi(\theta^{i}) \kappa^{i} \nu^{i} + \lambda^{i} \tau^{i})|_{(t=0,x=0)}= (\psi(\theta^{j}) \kappa^{j} \nu^{j} + \lambda^{j} \tau^{j})|_{(t=0,x=0)} \quad \text{ for }  i,j \in \{1,2,3\}.
\end{align}
First of all construct a diffeomorphism $\phi_{0}:[0,1] \to [0,1]$ such that $\phi_{0}(0)=0$, $\phi_{0}(1)=1$, $\phi_{0,y}>0$ in $[0,1]$ and 
\begin{align}\label{condphizero}
0=\frac{\psi(\theta (0,y))}{|u_{x}(0, y)|^{2}}\frac{\phi_{0,yy}(y)}{(\phi_{0,y}(y))^{2}}
+ \psi(\theta(0,y)) \frac{u_{xx}(0,y) \cdot u_{x} (0, y)}{|u_{x}(0, y)|^{4}} -\frac{\lambda(0, y)}{|u_{x}(0, y)|}
\end{align}
at $y=0,1$ (whereby recall that $\lambda(0,1)=0$). This can be done for instance by imposing also that $\phi_{0,y}(y)=1$ at $y=0,1$, and by  taking a suitable perturbation (near the boundary points)  of the identity map.
Next, note that at a boundary point $y =0,1$ we have 
\begin{align*}
\frac{\tilde{u}_{yy}}{|\tilde{u}_{y}|^{2}}(t,y)&= \frac{u_{xx} (t,y) (\phi_{y}(t,y))^{2} + u_{x}(t,y) \phi_{yy}(t,y)}{|u_{x}(t,y)|^{2} (\phi_{y}(t,y))^{2}} 
=\frac{u_{xx}(t,y)}{|u_{x}(t,y)|^{2}}+ \frac{\phi_{yy}(t,y)}{|u_{x}(t,y)| (\phi_{y}(t,y))^{2}}\tau(t,y)\\
&=(\kappa \nu)(t,y) + \left( \frac{u_{xx} (t,y)}{|u_{x}(t,y)|^{2}} \cdot \tau(t,y)
+  \frac{\phi_{yy}(t,y)}{|u_{x}(t,y)| (\phi_{y}(t,y))^{2}}\right)\tau(t,y)
\end{align*}
therefore by \eqref{primacond}, \eqref{condphizero}, and \eqref{boundmpiccolo} we infer that $\frac{\tilde{u}_{yy}}{|\tilde{u}_{y}|^{2}}(0,1)=0$ that is \eqref{cc2} is fulfilled.
Similarly using \eqref{secondacond} and \eqref{condphizero} we infer that \eqref{cc3} is also fulfilled.
Since 
\begin{align*}
\tilde{u}_{t}(t,y)&= u_{t}(t, \phi(t,y)) + u_{x}(t, \phi(t,y) )\phi_{t}(t,y)\\
&=(\psi(\theta) \kappa \nu  + \lambda \tau)(t, \phi(t,y)) + \phi_{t}(t,y) | u_{x}(t, \phi(t,y))| \tau(t, \phi(t,y)) \\
& = (\psi(\tilde{\theta}) \tilde{\kappa} \tilde{\nu})(t,y) + \big(\lambda(t, \phi(t,y))+ \phi_{t}(t,y) | u_{x}(t, \phi(t,y))| \big) \tau(t, \phi(t,y))  
\end{align*}
we see that for $\tilde{u}$ to fulfill \eqref{SpecialFlow}, we need $\phi$ to be a solution of
\begin{align*}
\phi_{t}(t,y)&=\frac{1}{|u_{x}(t, \phi(t,y))|} \left( \psi(\tilde{\theta}) (\frac{\tilde{u}_{yy}}{|\tilde{u}_{y}|^{2}} \cdot \tilde{\tau}) (t,y) -\lambda (t, \phi(t,y))\right)\\
&= \frac{\psi(\tilde{\theta})}{|u_{x}(t, \phi(t,y))|^{2}}\frac{\phi_{yy}(t,y)}{(\phi_{y}(t,y))^{2}}
+ \psi(\tilde{\theta}) \frac{u_{xx}(t, \phi(t,y)) \cdot u_{x} (t, \phi(t,y))}{|u_{x}(t, \phi(t,y))|^{4}} -\frac{\lambda(t, \phi(t,y))}{|u_{x}(t, \phi(t,y))|}
\end{align*}
(where $\psi(\tilde{\theta})=\varphi^{\circ}(\tilde{\nu}) D^{2} \varphi^{\circ} (\tilde{\nu}) \tilde{\tau} \cdot \tilde{\tau}$ with $\tilde{\tau}(t,y)= \tau(t, \phi(t,y))$) together with
$$\phi(t,0)=0, \quad \phi(t,1)=1, \quad \phi_{y}(t,y) >0 \quad \forall \, t \text{ and }\, y\in [0,1], $$
and 
$$ \phi(0, \cdot)=\phi_{0}(\cdot). $$
Observe that by the construction of $\phi_{0}$ the compatibility conditions of order zero are fulfilled. 
Instead of solving the PDE for $\phi$, it is convenient to work with the inverse diffeomorphism $\eta=\eta(t,x)$, such that  $\phi(t, \eta(t,x))=x$, and derive its existence first (as proposed in \cite{GoeMP}).
Indeed we see that $\eta$ must solve the \emph{linear} PDE
\begin{align*}
\eta_{t}(t,x) &=-\frac{\phi_{t}(t,y)}{\phi_{y}(t,y)}=\frac{\psi(\theta)}{|u_{x}(t,x)|^{2}} \eta_{xx}(t,x) -\eta_{x}(t,x) \left(  \psi(\theta) \frac{u_{xx}(t,x) \cdot u_{x}(t,x)}{|u_{x}(t,x)|^{4}} -\frac{\lambda(t,x)}{|u_{x}(t,x)|}\right)
\end{align*}
together with
$$ \eta(t,0) =0, \quad \eta(t,1)=1, \qquad \eta_{x}(t,x)>0\quad \forall \, t \text{ and }\, x\in [0,1], $$
and 
$$ \eta(0, \cdot)=\phi_{0}^{-1}(\cdot). $$
The existence of $\eta \in C^{\frac{2+\alpha}{2}, 2+\alpha} ([0,T] \times [0,1], \R) $ follows from standard theory \cite{Sol}. Possibly making the time interval smaller we can ensure that $\eta(t, \cdot)$ is a diffeomorphism. Finally we take $\phi(t, \cdot)= \eta^{-1}(t, \cdot)$.
\end{proof}

From Lemma \ref{lemma:diffeo}, Theorem \ref{teo:STE} and Corollary \ref{cor:STE} we directly obtain the following result.

\begin{theo}\label{teo:GEO}
Let $\alpha \in (0,1)$, $P^{i} \in \R^{2}$, $i=1,2,3$, be given points and $\sigma^{i}$, $i=1,2,3$, 
as in Definition~\ref{admtriod}. Then there exists $T>0$ and regular maps
$u^{i} \in C^{\frac{2+\alpha}{2}, 2+\alpha}([0,T) \times [0,1], \R^{2})\cap C^{\infty}((0,T) \times [0,1], \R^{2})$, 
$i=1,2,3$, which solve the geometric problem with initial conditions $u^{i}(0, x)=\sigma^{i}(x)$, $x \in [0,1]$, 
in the sense of Definition~\ref{geomsol} (i.e., up to reparametrization of the given initial data).
Moreover, the solutions $u^i$ are unique up to reparametrization, that is,  they parametrize a geometrically unique evolving network.
\end{theo}

We eventually show that at the maximal existence time either the length of one curve goes to zero or the $H^1$-norm 
of the curvature blows up.

\begin{prop}\label{pro:GEO}
Let $T$ be the maximal time such that there exist solutions of the geometric problem as in Theorem \ref{teo:GEO}, then we have
\begin{equation}\label{eqminmax}
\liminf_{t\to T} \min_{i\in\{1,2,3\}} L(u^i(t)) = 0
\qquad\text{or}\qquad
\limsup_{t\to T} \max_{i\in\{1,2,3\}} \| \kappa^i_\varphi\|_{H^1(I)} = +\infty.
\end{equation}
\end{prop}

\begin{proof}
Assume by contradiction that  $L(u^i(t))\ge\delta$ and $ \| \kappa^i_\varphi\|_{H^1(I)} \le C$, for all $i=1,2,3$ and $t\in [0,T)$, and for some $\delta,C>0$.
By Lemma \ref{lempara},
for any $\eps\in (0,T)$ we can reparametrize the admissible network $u^i(\cdot, T-\eps)$, $i=1,2,3$,  
in such a way that the reparametrizated network  $\sigma_\eps^i$
satisfy the compatibility conditions \eqref{cc1}, \eqref{cc2}, \eqref{cc3} and moreover
 $$
 \|\sigma_\eps^i\|_{C^{2,\frac 12}(I)}\le C', \qquad \|(\sigma_\eps^i)_x\|_{L^\infty(I)}\ge \frac{1}{C'},
 $$
where the constant $C'>0$ depends only on $\delta$ and $C$. Indeed, we can first reparametrize $u^i(\cdot, T-\eps)$ by constant speed. 
Then we notice that for the so obtained parametrization $v^{i}$ the uniform bound on the (anisotropic) curvature yields that 
$$\|\frac{v^{i}_{xx}}{|v_{x}^{i}|^{2}}\|_{C^{0,1/2}}= \|v^{i}_{xx} (\mathcal{L}(v^{i}))^{2}\|_{C^{0,1/2}} \leq C.$$ 
For  the compatibility conditions \eqref{cc1}, \eqref{cc2}, \eqref{cc3}  to hold we need now to reparametrize $v^{i}$ again (as explained in \eqref{condphizero} with $v$ instead of $u$, so that $|v_{x}|=1/\mathcal{L}(v^{i})$ and $v_{xx} \cdot v_{x}=0$). As appropriate diffeomorphisms $\phi^{i}$  we  take now   suitable perturbations near the junction of the identity map  such that $(\phi^{i})'(0)=(\phi^{i})'(1)=1$, $(\phi^{i})'>0$ on $[0,1]$, \eqref{condphizero} holds,  and the $\| \phi^{i}\|_{C^{2,1/2}}$-norm  is uniformly bounded by a constant depending only on  $C$, $\delta$, $\mathcal{L}(u^i( T-\eps))$, and the anisotropy map (see \eqref{condphizero} and recall \eqref{boundlbdry}, \eqref{boundmpiccolo}, Lemma~\ref{lempara}). The maps $\sigma^{i}_{\epsilon}= v^{i}(\phi^{i})$ satisfy the claims.

Then, by Theorem \ref{teo:STE} there exist solutions $u^i_\eps$ to the special flow starting from $\sigma_\eps^i$ at $T-\eps$,
defined on the time interval $[T-\eps,T-\eps+\tau)$,
where $\tau>0$ depends only on $\delta$ and $C'$ (in particular it is independent of $\eps$). 
By choosing $\eps$ small enough we then have $T-\eps+\tau>T$. 
Notice that, by Lemma \ref{lemma:diffeo} (see also Corollary \ref{cor:STE}) 
there exist smooth diffeomorphisms $\phi^i_\eps:(a,b)\times [0,1]\to [0,1]$, $(a,b) \subset (T-\eps,T)$ such that 
$u_\eps^i = u^i\circ \phi^i_\eps$, $i=1,2,3$. Let now $\eta\in C^\infty(\R)$ be such that $0\le \eta(t)\le1$ for all $t$,
$\eta(t)=0$  for $t\le a$ and $\eta(t)=1$  for all $t\ge b$, with $a<b$ and $[a,b]\subset (T-\eps,T)$,
then the functions
\[
\tilde u^i(t,x) = \left\{\begin{array}{lll}
u^i(t,x) && \text{for }(t,x)\in [0,a]\times[0,1]
\\
u^i(t,(1-\eta(t))x+\eta(t)\phi^i_\eps(t,x)) && \text{for }(t,x)\in (a,b)\times[0,1]
\\
u_\eps^i(t,x) && \text{for }(t,x)\in [b,T-\eps+\tau)\times[0,1]
\end{array}\right.
\]
give rise to geometric solution defined on the time interval $[0,T-\eps+\tau)$, contradicting the maximality of $T$.
\end{proof}


\section{Integral estimates and main result}\label{sec:integralEst}

In this section we derive integral estimates for a solution of the geometric problem (recall Section~\ref{sec:geoprob}).
We shall always assume that the flow is smooth up to the initial time $t=0$, which is not restrictive in view of Theorem \ref{teo:GEO}.  

We start with a general lemma.
\begin{lemma}\label{structurelemma}
Let $u:I\to\R^{2}$ satisfy \eqref{acsfT} for some smooth map $\lambda$. Let $S:I\to \R^{2}$ be a normal vector field along the curve $u$, that is $(S \cdot \tau)\equiv 0$.
Then
\begin{align}\label{structure}
\frac{d}{dt} &\left ( \frac{1}{2} \int_{I} |S|^{2} \frac{1}{\varphi^{\circ} (\nu)}  ds\right) + \int_{I} |S_{s}|^{2} \frac{\psi (\theta)}{\varphi^{\circ}(\nu)} ds \\
& = \left[ (S \cdot S_{s}) \frac{\psi (\theta)}{\varphi^{\circ}(\nu) }+\frac{1}{2}|S|^{2}\frac{\lambda}{\varphi^{\circ}(\nu)}  \right]_{0}^{1} \notag \\
& \qquad +\int_{I} (S \cdot (S_{t} -\psi(\theta) S_{ss}))  \frac{1}{\varphi^{\circ} (\nu)}  ds 
-\int_{I} (S \cdot S_{s}) \frac{\lambda}{\varphi^{\circ}(\nu)} ds  \notag\\
& \qquad- \int_{I} (S \cdot S_{s}) \frac{(\psi (\theta))_{s}}{\varphi^{\circ}(\nu)}  ds 
- \frac{1}{2} \int_{I} |S|^{2}  \frac{\psi(\theta)\kappa^{2}}{\varphi^{\circ} (\nu)}   ds \notag \\
& \qquad 
+\int_{I} \frac{D \varphi^{\circ} (\nu) \cdot \tau}{(\varphi^{\circ} (\nu))^{2}} \left( \frac{1}{2} |S|^{2} (\psi(\theta) \kappa)_{s}
 -\psi(\theta) \kappa (S \cdot S_{s}) \right) ds \notag
\end{align}
\end{lemma}
\begin{proof}
Since $$(ds)_{t}= (\lambda_{s} - (\kappa \nu \cdot u_{t})) ds =(\lambda_{s} -\psi(\theta)\kappa^{2}) ds$$ a direct computation gives
\begin{align*}
\frac{d}{dt} &\left ( \frac{1}{2} \int_{I} |S|^{2} \frac{1}{\varphi^{\circ} (\nu)}  ds\right) + \int_{I} |S_{s}|^{2} \frac{\psi (\theta)}{\varphi^{\circ}(\nu)} ds \\
& =\int_{I} (S \cdot (S_{t} -\psi S_{ss}))  \frac{1}{\varphi^{\circ} (\nu)}  ds + \frac{1}{2} \int_{I} |S|^{2}  \left(\frac{1}{\varphi^{\circ} (\nu)}  ds \right)_{t} +\left[ (S \cdot S_{s}) \frac{\psi (\theta)}{\varphi^{\circ}(\nu)}  \right]_{0}^{1}
\\
& \qquad - \int_{I} (S \cdot S_{s}) \left(\frac{\psi (\theta)}{\varphi^{\circ}(\nu)} \right)_{s} ds\\
&=\int_{I} (S \cdot (S_{t} -\psi(\theta) S_{ss}))  \frac{1}{\varphi^{\circ} (\nu)}  ds + \frac{1}{2} \int_{I} |S|^{2}  \left(\frac{1}{\varphi^{\circ} (\nu)}   \right)_{t} ds \\
&\qquad +\left[ (S \cdot S_{s}) \frac{\psi (\theta)}{\varphi^{\circ}(\nu) }+\frac{1}{2}|S|^{2}\frac{\lambda}{\varphi^{\circ}(\nu)}  \right]_{0}^{1}
\\
& \qquad - \int_{I} (S \cdot S_{s}) \left(\frac{\psi (\theta)}{\varphi^{\circ}(\nu)} \right)_{s} ds
- \frac{1}{2} \int_{I} |S|^{2}  \frac{\psi(\theta)\kappa^{2}}{\varphi^{\circ} (\nu)}   ds \\
& \qquad 
-\int_{I} (S \cdot S_{s}) \frac{\lambda}{\varphi^{\circ}(\nu)} ds 
-\frac{1}{2} \int_{I} |S|^{2} \lambda \left( \frac{1}{\varphi^{\circ}(\nu)} \right)_{s} ds. 
 \end{align*}
 Using the expression for $\theta_{t}$ from Lemma~\ref{lemma2.1} we observe
 \begin{align*}
 \frac{1}{2}& |S|^{2} \left( \frac{1}{\varphi^{\circ} (\nu)}\right)_{t} - (S \cdot S_{s}) \psi (\theta)\left( \frac{1}{\varphi^{\circ} (\nu)}\right)_{s} -\frac{1}{2} |S|^{2} \lambda \left( \frac{1}{\varphi^{\circ} (\nu)}\right)_{s}\\
 & =\frac{D \varphi^{\circ} (\nu) \cdot \tau}{(\varphi^{\circ} (\nu))^{2}} \left( \frac{1}{2} |S|^{2} (\psi(\theta) \kappa)_{s}
 -\psi(\theta) \kappa (S \cdot S_{s}) \right)
 \end{align*}
 and the claim follows.
\end{proof}

Also we recall some useful interpolation estimates (here we cite \cite[Proposition~3.11, Remark~3.12]{MNT}):
\begin{prop}\label{IEst}
Let $u$ be a smooth regular curve in $\R^{2}$ with finite length $L$. If $f$ is a smooth function defined on $u$ and $m \geq 1$, $p \in [2, +\infty]$, we have the estimates
\begin{align*}
\|\partial_{s}^{n } f \|_{ L^{p}} \leq C_{n,m,p} \| \partial_{s}^{m} f \|_{L^{2}}^{\sigma} \| f \|_{L^{2}}^{1-\sigma} +\frac{B_{n,m,p}}{L^{m\sigma}} \| f \|_{L^{2}}
\end{align*}
for every $n \in \{0, \ldots, m-1 \}$ where $\sigma= \frac{n+1/2 -1/p}{m}$ and the constants $C_{n,m,p}$, $B_{n,m,p}$ are independent of $u$. In particular
\begin{align*}
\|\partial_{s}^{n } f \|_{ L^{\infty}} \leq C_{n,m} \| \partial_{s}^{m} f \|_{L^{2}}^{\sigma} \| f \|_{L^{2}}^{1-\sigma} +\frac{B_{n,m}}{L^{m\sigma}} \| f \|_{L^{2}} \qquad \text{ with } \sigma=\frac{n+1/2}{m}.
\end{align*}
\end{prop}

\subsection{Estimates on $\|\kappa\|_{L^{2}}$ and $\| \kappa_{\varphi}\|_{L^{2}}$}
We now apply the  Lemma~\ref{structurelemma} for the special choice of $S = \psi(\theta)\kappa \nu$, which is the normal component of the velocity vector. 
Using Lemma~\ref{lemma2.1} we compute (here and below we write $w^{\perp}=(w\cdot \nu) \nu$ for the normal component of a vector $w\in \R^{2}$)
\begin{align*}
S&= (u_{t})^{\perp} = \psi(\theta)\kappa \nu,   & |S|^{2} =  (\psi(\theta)\kappa)^{2},\\
S_{s}& =(\psi(\theta)\kappa)_{s} \nu - (\psi(\theta)\kappa) \kappa \tau,   & |S_s|^{2} =  ((\psi(\theta)\kappa)_{s})^{2} + (\psi(\theta)\kappa)^{2} \kappa^{2},
\end{align*}
as well as
\begin{align*}
S_{ss}&=     ((\psi(\theta)\kappa)_{ss} - (\psi(\theta)\kappa) \kappa^{2}) \nu + (\ldots) \tau, \\
S_{t}&= [ \psi'(\theta) ( (\psi(\theta) \kappa)_{s} +\lambda \kappa)\kappa +\psi(\theta) ( (\psi(\theta) \kappa)_{ss}+ \psi(\theta)\kappa^{3} + \lambda \kappa_{s}) ]   \nu  + (\ldots) \tau\\
(S \cdot ( S_{t} -\psi(\theta)S_{ss})) &= \psi(\theta) \psi'(\theta) \kappa^{2} (\psi(\theta) \kappa)_{s} + 2 (\psi(\theta)\kappa)^{3} \kappa + \lambda \frac{1}{2} (|S^{2}|)_{s}\\
&= (\psi(\theta))_{s} (S \cdot S_{s}) + 2 (\psi(\theta)\kappa)^{3} \kappa + \lambda  (S \cdot S_{s}).
\end{align*}
Therefore the integral terms appearing in the right hand-side of ~\eqref{structure} amount to
\begin{align*}
&\int_{I} (S \cdot (S_{t} -\psi(\theta) S_{ss}))  \frac{1}{\varphi^{\circ} (\nu)}  ds 
-\int_{I} (S \cdot S_{s}) \frac{\lambda}{\varphi^{\circ}(\nu)} ds \\
& \qquad- \int_{I} (S \cdot S_{s}) \frac{(\psi (\theta))_{s}}{\varphi^{\circ}(\nu)}  ds
- \frac{1}{2} \int_{I} |S|^{2}  \frac{\psi(\theta)\kappa^{2}}{\varphi^{\circ} (\nu)}   ds \\
& \qquad 
+\int_{I} \frac{D \varphi^{\circ} (\nu) \cdot \tau}{(\varphi^{\circ} (\nu))^{2}} \left( \frac{1}{2} |S|^{2} (\psi(\theta) \kappa)_{s}
 -\psi(\theta) \kappa (S \cdot S_{s}) \right) ds\\
 &=\int_{I}  \frac{3}{2} \frac{(\psi(\theta)\kappa)^{3} \kappa }{\varphi^{\circ}(\nu)}  ds -\frac{1}{2}\int_{I} \frac{D \varphi^{\circ} (\nu) \cdot \tau}{(\varphi^{\circ} (\nu))^{2}} (\psi(\theta)\kappa)^{2} (\psi(\theta)\kappa)_{s} ds.
\end{align*}
In particular notice that with $S_{t} -\psi(\theta) S_{ss}$ high order terms disappear.
Equation \eqref{structure} becomes
\begin{align}\label{equaz31}
\frac{d}{dt} &\left ( \frac{1}{2} \int_{I}  \frac{(\psi(\theta)\kappa)^{2}}{\varphi^{\circ} (\nu)}  ds\right) + \int_{I} |(\psi(\theta)\kappa)_{s}|^{2} \frac{\psi (\theta)}{\varphi^{\circ}(\nu)} ds 
=\left[ (S \cdot S_{s}) \frac{\psi (\theta)}{\varphi^{\circ}(\nu) }+\frac{1}{2}|S|^{2}\frac{\lambda}{\varphi^{\circ}(\nu)}  \right]_{0}^{1}\\
&\quad
+\int_{I}  \frac{1}{2} \frac{(\psi(\theta)\kappa)^{3} \kappa }{\varphi^{\circ}(\nu)}  ds -\frac{1}{2}\int_{I} \frac{D \varphi^{\circ} (\nu) \cdot \tau}{(\varphi^{\circ} (\nu))^{2}} (\psi(\theta)\kappa)^{2} (\psi(\theta)\kappa)_{s} ds.
\end{align}
For the boundary term we notice  that
\begin{align*}
(S \cdot S_{s}) \frac{\psi (\theta)}{\varphi^{\circ}(\nu) } = (S \cdot  (S_{s})^{\perp}) \frac{\psi (\theta)}{\varphi^{\circ}(\nu) } 
=(u_{t} \cdot  (S_{s})^{\perp}) \frac{\psi (\theta)}{\varphi^{\circ}(\nu) } .
\end{align*}
This motivates the choice of $S$ since at the boundary the velocity $u_{t}$ is either zero (at the fixed boundary point) or coincides with the velocity of the other curves meeting at the triple junction. We can then lower the order of the  terms at the moving  boundary point  by exploiting the boundary conditions. More precisely we write
\begin{align*}
(S \cdot S_{s}) \frac{\psi (\theta)}{\varphi^{\circ}(\nu) }& =(u_{t} \cdot (\psi(\theta)\kappa)_{s} \nu ) \frac{\psi (\theta)}{\varphi^{\circ}(\nu) } = (u_{t} \cdot  ((\psi(\theta)\kappa)_{s} +\lambda \kappa )\nu ) \frac{\psi (\theta)}{\varphi^{\circ}(\nu) } 
- (u_{t} \cdot  \lambda \kappa \nu ) \frac{\psi (\theta)}{\varphi^{\circ}(\nu) } \\
& = (u_{t} \cdot  \frac{\psi (\theta) \theta_{t}}{\varphi^{\circ}(\nu) } \nu) -  \lambda \psi(\theta) \kappa^{2}  \frac{\psi (\theta)}{\varphi^{\circ}(\nu) } 
\end{align*}
which yields
\begin{align*}
\left[ (S \cdot S_{s}) \frac{\psi (\theta)}{\varphi^{\circ}(\nu) }+\frac{1}{2}|S|^{2}\frac{\lambda}{\varphi^{\circ}(\nu)}  \right]_{0}^{1} =
 \left [ (u_{t} \cdot  \frac{\psi (\theta) \theta_{t}}{\varphi^{\circ}(\nu) } \nu)  
  -\frac{1}{2} \lambda 
\frac{(\psi (\theta) \kappa)^{2}}{\varphi^{\circ}(\nu) }\right]_{0}^{1}.
\end{align*}

Derivation in time of \eqref{HC} gives (at the junction)
\begin{align*}
0 =-\sum_{i=1}^{3} (D \varphi^{\circ}(\nu^{i}))_{t} = \sum_{i=1}^{3} D^{2}\varphi^{\circ}(\nu^{i}) \theta^{i}_{t} \tau^{i}
=\sum_{i=1}^{3} \psi(\theta^{i}) \theta^{i}_{t}\frac{\tau^{i}}{\varphi^{\circ}(\nu^{i})}.
\end{align*}
After multiplication with $\tiny{\left(\begin{array}{cc} 0&-1\\1&0 \end{array} \right)}$, which rotates vectors by $\pi/2$, we finally infer that
\begin{align*}
0=\sum_{i=1}^{3} \psi(\theta^{i}) \theta^{i}_{t}\frac{\nu^{i}}{\varphi^{\circ}(\nu^{i})}
\end{align*}
holds at the junction point.
In particular, since   $u_{t}^{1}=u_{t}^{2}=u_{t}^{3}$ at the junction point, we infer
\begin{align*}
\sum_{i=1}^{3}(u_{t}^{i} \cdot  \frac{\psi (\theta^{i}) \theta_{t}^{i}}{\varphi^{\circ}(\nu^{i}) } \nu^{i})  =
(u_{t}^{i} \cdot \sum_{i=1}^{3} \frac{\psi (\theta^{i}) \theta_{t}^{i}}{\varphi^{\circ}(\nu^{i}) } \nu^{i})  =0.
\end{align*}
Summing \eqref{equaz31} for every curve in the network, we therefore obtain
\begin{align}\label{erstestufe}
\sum_{i=1}^{3} \frac{d}{dt} &\Big( \frac{1}{2} \int_{I}  \frac{(\psi(\theta^{i})\kappa^{i})^{2}}{\varphi^{\circ} (\nu^{i})}  ds\Big) +  \sum_{i=1}^{3}\int_{I} |(\psi(\theta^{i})\kappa^{i})_{s}|^{2} \frac{\psi (\theta^{i})}{\varphi^{\circ}(\nu^{i})} ds   
= \sum_{i=1}^{3}  \frac{1}{2} \lambda^{i} 
\frac{(\psi (\theta^{i}) \kappa^{i})^{2}}{\varphi^{\circ}(\nu^{i})} \Big|_{x=0}  \\
&\quad +\sum_{i=1}^{3}
\left( \int_{I}  \frac{1}{2} \frac{(\psi(\theta^{i})\kappa^{i})^{3} \kappa^{i} }{\varphi^{\circ}(\nu^{i})}  ds -\frac{1}{2}\int_{I} \frac{D \varphi^{\circ} (\nu^{i}) \cdot \tau^{i}}{(\varphi^{\circ} (\nu^{i}))^{2}} (\psi(\theta^{i})\kappa^{i})^{2} (\psi(\theta^{i})\kappa^{i})_{s} ds \right). \notag
\end{align}
A more geometrical interpretation of the above expression is discussed in Remark~\ref{vameglio?} below.
Using \eqref{boundmpiccolo} as well as
$ C^{-1} \leq \varphi^{\circ} (\nu)\leq C$ and  $|D\varphi^{\circ} (\nu)|\leq C$ (recall that $D\varphi^{\circ} (\nu)$ lies on the Wulff shape) we can write
\begin{align*}
\sum_{i=1}^{3}  \frac{d}{dt} &\left( \frac{1}{2} \int_{I}  \frac{(\psi(\theta^{i})\kappa^{i})^{2}}{\varphi^{\circ} (\nu^{i})}  ds\right) + m\int_{I} |(\psi(\theta^{i})\kappa^{i})_{s}|^{2} \frac{1}{\varphi^{\circ}(\nu^{i})} ds \\
&\leq \sum_{i=1}^{3} C |\lambda^{i}| (\psi(\theta^{i}) \kappa^{i})^{2} \Big|_{x=0}  +\sum_{i=1}^{3}
\left( C_{\epsilon}\int_{I}  (\psi(\theta^{i})\kappa^{i})^{4}   ds + \epsilon  \int_{I} |(\psi(\theta^{i})\kappa^{i})_{s}|^{2} \frac{1}{\varphi^{\circ}(\nu^{i})} ds \right).
\end{align*}
Choosing $\epsilon=m/2$ we obtain
\begin{align*}
\sum_{i=1}^{3}  \frac{d}{dt} &\left( \frac{1}{2} \int_{I}  \frac{(\psi(\theta^{i})\kappa^{i})^{2}}{\varphi^{\circ} (\nu^{i})}  ds\right) + \frac{m}{2}\int_{I} |(\psi(\theta^{i})\kappa^{i})_{s}|^{2} \frac{1}{\varphi^{\circ}(\nu^{i})} ds \\
&\leq \sum_{i=1}^{3} C |\lambda^{i}| ( \psi(\theta^{i})\kappa^{i})^{2} \Big|_{x=0}  +\sum_{i=1}^{3} C\|\psi(\theta^{i})\kappa^{i}\|_{L^{4}(I)}^{4} .
\end{align*}
Next we apply interpolation estimates, under the assumption that we have a uniform control (from below) of the lengths of the curves composing the network.
By Proposition~\ref{IEst} it follows
\begin{align*}
\|\psi(\theta^{i})\kappa^{i}\|_{L^{4}}^{4} &\leq C( \| (\psi(\theta^{i})\kappa^{i})_{s} \|_{L^{2}}^{1/4} \| \psi(\theta^{i})\kappa^{i} \|_{L^{2}}^{3/4} +  \| \psi(\theta^{i})\kappa^{i} \|_{L^{2}} )^{4} \\
&\leq  C \| (\psi(\theta^{i})\kappa^{i})_{s} \|_{L^{2}} \| \psi(\theta^{i})\kappa^{i} \|_{L^{2}}^{3} + C \| \psi(\theta^{i})\kappa^{i} \|_{L^{2}}^{4}\\
& \leq \epsilon \int_{I} |(\psi(\theta^{i})\kappa^{i})_{s}|^{2} \frac{1}{\varphi^{\circ}(\nu^{i})} ds  + C_{\epsilon} \left( \int_{I}  \frac{(\psi(\theta^{i})\kappa^{i})^{2}}{\varphi^{\circ} (\nu^{i})}  ds\right)^{3} +C.
\end{align*}
Moreover, for the boundary term we use \eqref{boundlbdry} and Proposition~\ref{IEst} to infer
\begin{align*}
|\lambda^{i}| (\psi(\theta^{i})\kappa^{i})^{2} \Big|_{x=0} &\leq  C(\psi(\theta^{i})\kappa^{i})^{2} \sum_{j=1}^{3} |\psi(\theta^{j})\kappa^{j}| \Big|_{x=0} \leq C \| \psi(\theta^{i})\kappa^{i}\|_{L^{\infty}}^{2}\sum_{j=1}^{3}\|\psi(\theta^{j}) \kappa^{j}\|_{L^{\infty}}\\
& \leq C (\|(\psi(\theta^{i})\kappa^{i})_{s} \|_{L^{2}}^{1/2} \|\psi(\theta^{i})\kappa^{i} \|_{L^{2}}^{1/2} +   \|\psi(\theta^{i})\kappa^{i} \|_{L^{2}})^{2} \\
 & \qquad \cdot ( \sum_{j=1}^{3}(\|(\psi(\theta^{j})\kappa^{j})_{s} \|_{L^{2}}^{1/2} \|\psi(\theta^{j})\kappa^{j} \|_{L^{2}}^{1/2} +   \|\psi(\theta^{j})\kappa^{j} \|_{L^{2}}))\\
 & \leq \epsilon \sum_{j=1}^{3} \int_{I} |(\psi(\theta^{j})\kappa^{j})_{s}|^{2} \frac{1}{\varphi^{\circ}(\nu^{j})} ds  + C_{\epsilon} \sum_{j=1}^{3} \left( \int_{I}  \frac{(\psi(\theta^{j})\kappa^{j})^{2}}{\varphi^{\circ} (\nu^{j})}  ds\right)^{3} +C_{\epsilon},
\end{align*}
where for the last step,  we have used several times the Young-inequality.
Putting all estimates together and choosing $\epsilon$ appropriately we infer
\begin{align*}
\sum_{i=1}^{3}  \frac{d}{dt} \left( \frac{1}{2} \int_{I}  \frac{(\psi(\theta^{i})\kappa^{i})^{2}}{\varphi^{\circ} (\nu^{i})}  ds\right)  
&\leq\sum_{i=1}^{3} C \left( \int_{I}  \frac{(\psi(\theta^{i})\kappa^{i})^{2}}{\varphi^{\circ} (\nu^{i})}  ds\right)^{3} +C\\
&\leq C \sum_{i=1}^{3}\left(1+ \int_{I}  \frac{(\psi(\theta^{i})\kappa^{i})^{2}}{\varphi^{\circ} (\nu^{i})}  ds \right)^{3},
\end{align*}
where $C$ depends on the anisotropy (precisely \eqref{boundmpiccolo}, $a_{0}$ as in \eqref{a0}, as well as
$ C^{-1} \leq \varphi^{\circ} (\nu)\leq C$ and  $|D\varphi^{\circ} (\nu)|\leq C$) and on the uniform bound from below 
on the lengths of the curves. Note also that so far only  information of $\lambda^{i}$ at the boundary has played a role.
 Recalling that $\kappa_{\varphi}=\frac{\psi(\theta)}{\varphi^{\circ}(\nu)} \kappa$ and integration in time  for $0 \leq t_{1} < t_{2}$ yields
\begin{align*}
-\frac{1}{(\sum_{i=1}^{3} (1+ \int_{I}  (\kappa_{\varphi}^{i})^{2}\varphi^{\circ} (\nu^{i})  ds) )^{2}  |_{t=t_{2}}} 
+\frac{1}{(\sum_{i=1}^{3} (1+ \int_{I}  (\kappa_{\varphi}^{i})^{2}\varphi^{\circ} (\nu^{i})  ds ))^{2} |_{t=t_{1}}}  \leq C (t_{2}-t_{1}).
\end{align*}
In particular if, for $ 0<T <\infty$, there exists a sequence of times $t_{j} \to T$, for $j \to \infty$,  such that
\begin{align}\label{blowupk}
\left(\sum_{i=1}^{3}\int_{I}(\kappa_{\varphi}^{i})^{2}\varphi^{\circ} (\nu^{i})  ds \right) \Big| _{t=t_{j}} \to \infty \qquad \text{as } j\to \infty
\end{align}
then we obtain for any $ t \in [0,T)$ that
\begin{align*} 
\frac{1}{(\sum_{i=1}^{3} (1+ \int_{I}  (\kappa_{\varphi}^{i})^{2}\varphi^{\circ} (\nu^{i})  ds ))^{2}}  \leq C (T-t).
\end{align*}
Therefore we can conclude with the following statement, that is valid for a solution to the geometric problem posed in Section~\ref{sec:geoprob}:
\begin{lemma} \label{lem5.2}
If for $ 0<T <\infty$,  the lengths of the curves are uniformly bounded from below 
\begin{align*}
L(u^{i}(t)) \geq \delta >0, \qquad i=1,2,3, \text{ for any } t \in [0,T)
\end{align*}
and there exists a sequence of times $t_{j} \to T$, for $j \to \infty$,  such that \eqref{blowupk} holds, then there exists a positive constant $C$ such that
$$ \sum_{i=1}^{3}  \int_{I}  (\kappa_{\varphi}^{i})^{2}\varphi^{\circ} (\nu^{i}) ds  \geq \frac{C}{\sqrt{T-t}}   
\qquad \text{for any $t\in [0,T)$,}$$
where the constant $C>0$ depends on $\delta$ and $\varphi^{\circ}$  (namely $m$, $M$ (recall \eqref{boundmpiccolo}), $a_{0}$ (recall \eqref{a0}),  $ C^{-1} \leq \varphi^{\circ} (\nu)\leq C$ and  $|D\varphi^{\circ} (\nu)|\leq C$). 
 In particular, taking $t=0$ we have
 $$
 \sqrt{T} \ge \frac{C}{\sum_{i=1}^{3}  \| \kappa_{\varphi}^{i}(0, \cdot)\|_{L^{2}(I)}^{2} }\,.
 $$ 
\end{lemma}

\begin{rem}\label{vameglio?}
Upon recalling that $\kappa_{\varphi}=\frac{\psi(\theta)}{\varphi^{\circ}(\nu)} \kappa$, observe that \eqref{erstestufe} can also be written as
\allowdisplaybreaks{
\begin{align*}
\sum_{i=1}^{3}  \frac{d}{dt} &\left( \frac{1}{2} \int_{I}  |\kappa_{\varphi}^{i}|^{2}\varphi^{\circ} (\nu^{i})  ds\right) + \sum_{i=1}^{3} \int_{I} |(\psi(\theta^{i})\kappa^{i})_{s}|^{2} \frac{\psi (\theta^{i})}{\varphi^{\circ}(\nu^{i})} ds  
= \sum_{i=1}^{3}  \frac{1}{2} \lambda^{i} 
(\kappa_{\varphi}^{i})^{2} \varphi^{\circ}(\nu^{i}) \Big|_{x=0}  \\
&\quad +\sum_{i=1}^{3}
\left( \int_{I}  \frac{1}{2} \frac{(\psi(\theta^{i})\kappa^{i})^{3} \kappa^{i} }{\varphi^{\circ}(\nu^{i})}  ds -\frac{1}{2}\int_{I} D \varphi^{\circ} (\nu^{i}) \cdot \tau^{i} (\kappa_{\varphi}^{i})^{2} (\psi(\theta^{i})\kappa^{i})_{s} ds \right)\\
&= \sum_{i=1}^{3}  \frac{1}{2} \lambda^{i} 
(\kappa_{\varphi}^{i})^{2} \varphi^{\circ}(\nu^{i}) \Big|_{x=0}  - \sum_{i=1}^{3}\int_{I} D \varphi^{\circ} (\nu^{i}) \cdot \tau^{i} (\kappa_{\varphi}^{i})^{2} (\psi(\theta^{i})\kappa^{i})_{s} ds \\
& \quad + \sum_{i=1}^{3}
\left( \int_{I}  \frac{1}{2} \frac{(\psi(\theta^{i})\kappa^{i})^{3} \kappa^{i} }{\varphi^{\circ}(\nu^{i})}  ds +\frac{1}{2}\int_{I} D \varphi^{\circ} (\nu^{i}) \cdot \tau^{i} (\kappa_{\varphi}^{i})^{2} (\psi(\theta^{i})\kappa^{i})_{s} ds \right)\\
&= \sum_{i=1}^{3}  \frac{1}{2} \left(  \lambda^{i} 
(\kappa_{\varphi}^{i})^{2} - D \varphi^{\circ} (\nu^{i}) \cdot \tau^{i} (\kappa_{\varphi}^{i})^{3} \right) \varphi^{\circ}(\nu^{i}) \Big|_{x=0}  \\
&\quad - \sum_{i=1}^{3}\int_{I} D \varphi^{\circ} (\nu^{i}) \cdot \tau^{i} (\kappa_{\varphi}^{i})^{2} (\psi(\theta^{i})\kappa^{i})_{s} ds \\
&\quad +\sum_{i=1}^{3} \Big( \int_{I}  \frac{1}{2} \frac{(\psi(\theta^{i})\kappa^{i})^{3} \kappa^{i} }{\varphi^{\circ}(\nu^{i})}  ds + \int_{I}  \frac{1}{2} (\kappa_{\varphi}^{i})^{4} \varphi^{\circ}(\nu^{i}) ds \\
& \quad - \int_{I} D \varphi^{\circ} (\nu^{i}) \cdot \tau^{i} (\kappa_{\varphi}^{i})^{2} (\kappa_{\varphi}^{i})_{s}  \varphi^{\circ}(\nu^{i}) ds -\frac{1}{2}\int_{I} D \varphi^{\circ} (\nu^{i}) \cdot \nu^{i} \kappa^{i} (\kappa_{\varphi}^{i})^{2} \psi(\theta^{i})\kappa^{i} ds
\Big)\\ %
&= \sum_{i=1}^{3}  \frac{1}{2} \left(  \lambda^{i} 
(\kappa_{\varphi}^{i})^{2} - D \varphi^{\circ} (\nu^{i}) \cdot \tau^{i} (\kappa_{\varphi}^{i})^{3} \right) \varphi^{\circ}(\nu^{i}) \Big|_{x=0}  \\
&\quad - \sum_{i=1}^{3}\int_{I} D \varphi^{\circ} (\nu^{i}) \cdot \tau^{i} (\kappa_{\varphi}^{i})^{2} (\psi(\theta^{i})\kappa^{i})_{s} ds \\%
&\quad +\sum_{i=1}^{3} \Big( \int_{I}  \frac{1}{2} (\kappa_{\varphi}^{i})^{4} \varphi^{\circ}(\nu^{i}) ds 
 - \int_{I} D \varphi^{\circ} (\nu^{i}) \cdot \tau^{i} (\kappa_{\varphi}^{i})^{2} (\kappa_{\varphi}^{i})_{s}  \varphi^{\circ}(\nu^{i}) ds 
\Big).
\end{align*}   }
where in the integration by parts we have used the fact that the velocities and hence the curvatures vanish at the fixed boundary points.
On the other hand note that 
\begin{align*}
(\kappa_{\varphi})_{s}= \frac{(\psi(\theta) \kappa)_{s}}{\varphi^{\circ}(\nu)} + \psi(\theta)\kappa \left( \frac{1}{\varphi^{\circ} (\nu)}\right)_{s} = \frac{(\psi(\theta) \kappa)_{s}}{\varphi^{\circ}(\nu)}  + \kappa_{\varphi}\kappa \frac{D\varphi^{\circ} (\nu) \cdot \tau}{\varphi^{\circ}(\nu)}
\end{align*}
and therefore
\begin{align*}
\frac{1}{2} &\int_{I} |(\psi(\theta^{i})\kappa^{i})_{s}|^{2} \frac{\psi (\theta^{i})}{\varphi^{\circ}(\nu^{i})} ds  \\
& = \frac{1}{2} \int_{I} \psi(\theta^{i}) |(\kappa^{i}_{\varphi})_{s}|^{2} \varphi^{\circ} (\nu^{i}) ds 
+\frac{1}{2} \int_{I} \left(\kappa^{i}_{\varphi}\kappa^{i} \frac{D\varphi^{\circ} (\nu^{i}) \cdot \tau^{i}}{\varphi^{\circ}(\nu^{i})}
\right)^{2}  \psi(\theta^{i}) \varphi^{\circ} (\nu^{i})ds\\
&\quad-\int_{I} (\kappa^{i}_{\varphi})_{s}(\kappa^{i}_{\varphi})^{2} (D\varphi^{\circ} (\nu^{i}) \cdot \tau^{i} ) \varphi^{\circ} (\nu^{i}) ds.
\end{align*}
It follows then
\begin{align*}
\sum_{i=1}^{3}   \Big \{ \frac{d}{dt} &\left( \frac{1}{2} \int_{I}  |\kappa_{\varphi}^{i}|^{2}\varphi^{\circ} (\nu^{i})  ds\right) + \frac{1}{2}\int_{I} |(\psi(\theta^{i})\kappa^{i})_{s}|^{2} \frac{\psi (\theta^{i})}{\varphi^{\circ}(\nu^{i})} ds  
+\frac{1}{2} \int_{I} \psi(\theta^{i}) |(\kappa^{i}_{\varphi})_{s}|^{2} \varphi^{\circ} (\nu^{i}) ds \\
&+\frac{1}{2} \int_{I} \left(\kappa^{i}_{\varphi}\kappa^{i} \frac{D\varphi^{\circ} (\nu^{i}) \cdot \tau^{i}}{\varphi^{\circ}(\nu^{i})}
\right)^{2}  \psi(\theta^{i}) \varphi^{\circ} (\nu^{i})ds  \Big \} \\
&=
\sum_{i=1}^{3}  \frac{1}{2} \left(  \lambda^{i} 
(\kappa_{\varphi}^{i})^{2} - D \varphi^{\circ} (\nu^{i}) \cdot \tau^{i} (\kappa_{\varphi}^{i})^{3} \right) \varphi^{\circ}(\nu^{i}) \Big|_{x=0}  \\
&\quad - \sum_{i=1}^{3}\int_{I} D \varphi^{\circ} (\nu^{i}) \cdot \tau^{i} (\kappa_{\varphi}^{i})^{2} (\psi(\theta^{i})\kappa^{i})_{s} ds  +\sum_{i=1}^{3}  \int_{I}  \frac{1}{2} (\kappa_{\varphi}^{i})^{4} \varphi^{\circ}(\nu^{i}) ds .
\end{align*}
For the second last integral on  the right-hand side note that 
\begin{align*}
\Big|\int_{I} &D \varphi^{\circ} (\nu^{i}) \cdot \tau^{i} (\kappa_{\varphi}^{i})^{2} (\psi(\theta^{i})\kappa^{i})_{s} ds \Big|  = 
\Big|  \int_{I}  D \varphi^{\circ} (\nu^{i}) \cdot \tau^{i} \kappa_{\varphi}^{i} \kappa^{i} \sqrt{\frac{\psi}{\varphi^{\circ}}} \sqrt{\frac{\psi}{\varphi^{\circ}}} (\psi(\theta^{i})\kappa^{i})_{s} ds \Big|\\
& \leq \frac{1}{2}\int_{I} |(\psi(\theta^{i})\kappa^{i})_{s}|^{2} \frac{\psi (\theta^{i})}{\varphi^{\circ}(\nu^{i})} ds  +
\frac{1}{2} \int_{I} \left(\kappa^{i}_{\varphi}\kappa^{i} \frac{D\varphi^{\circ} (\nu^{i}) \cdot \tau^{i}}{\varphi^{\circ}(\nu^{i})}
\right)^{2}  \psi(\theta^{i}) \varphi^{\circ} (\nu^{i})ds 
\end{align*}
and so it can be nicely absorbed. It follows then 
\begin{align*}
\sum_{i=1}^{3}   \frac{d}{dt} &\left( \frac{1}{2} \int_{I}  |\kappa_{\varphi}^{i}|^{2}\varphi^{\circ} (\nu^{i})  ds\right)  
+ \sum_{i=1}^{3} \frac{1}{2} \int_{I} \psi(\theta^{i}) |(\kappa^{i}_{\varphi})_{s}|^{2} \varphi^{\circ} (\nu^{i}) ds  \\
&\leq
\sum_{i=1}^{3}  \frac{1}{2} \left(  \lambda^{i} 
(\kappa_{\varphi}^{i})^{2} - D \varphi^{\circ} (\nu^{i}) \cdot \tau^{i} (\kappa_{\varphi}^{i})^{3} \right) \varphi^{\circ}(\nu^{i}) \Big|_{x=0}    +\sum_{i=1}^{3}  \int_{I}  \frac{1}{2} (\kappa_{\varphi}^{i})^{4} \varphi^{\circ}(\nu^{i}) ds .
\end{align*}
\end{rem}

\subsection{Estimates on $\|(\psi\kappa)_{ss}\|_{L^{2}}$ and $\| (\kappa_{\varphi})_{ss}\|_{L^{2}}$}

Similarly to \cite{MNT} (where the isotropic setting is considered) we now introduce some notation that  simplifies exposition and reading.
We indicate by $p_{\sigma}(\partial_{s}^{h} (\psi \kappa))$ a polynomial in the variables $(\psi(\theta) \kappa), \ldots,\partial_{s}^{h} (\psi(\theta) \kappa) $ with coefficient functions $C=C(\frac{1}{\psi}, \psi, \ldots, \partial_{\theta}^{h+1}\psi)$ that depend on $\frac{1}{\psi}, \psi, \ldots, \partial_{\theta}^{h+1}\psi$ 
and such that every monomial is of the form
\begin{align*}
C(\frac{1}{\psi}, \psi, \ldots, \partial_{\theta}^{h+1}\psi) \prod_{l=0}^{h} (\partial_{s}^{l} (\psi\kappa))^{\beta_{l}}
\qquad \text{ with } \sum_{l=0}^{h}(l+1)\beta_{l}=\sigma.
\end{align*}
Note that, due to the smoothness assumptions on the anisotropy map $\varphi^{\circ}$ we will be able to bound uniformly from above all  the coefficient maps $C(\frac{1}{\psi}, \psi, \ldots, \partial_{\theta}^{h+1}\psi) $, that is
$$ |C(\frac{1}{\psi}, \psi, \ldots, \partial_{\theta}^{h+1}\psi) | \leq C_{h}, \text{ for any } h \in \mathbb{N}_{0}. $$
 For this reason we treat these maps  as coefficients and refer to them as such.
More precisely the maps $C(\frac{1}{\psi}, \psi, \ldots, \partial_{\theta}^{h}\psi)$ are assumed to be sums of rational functions  of type
$$ \frac{\text{polynomial with constant coefficients in the  variables  $\psi(\theta), \ldots, \partial_{\theta}^{h}\psi(\theta) $ } }{\psi^{r}(\theta)}$$
for some $r \in \mathbb{N}_{0}$ and, as a consequence,  the following rule applies
\begin{align}\label{rule}
\partial_{s} \left( C(\frac{1}{\psi}, \psi, \ldots, \partial_{\theta}^{h}\psi) \right) = C(\frac{1}{\psi}, \psi, \ldots, \partial_{\theta}^{h+1}\psi) (\psi \kappa)
\end{align}
which is obtained by derivating the expression and  using $\theta_{s}=\kappa= \frac{1}{\psi} (\psi \kappa)$.

Similarly we indicate by 
$p_{\sigma}(|\partial_{s}^{h} (\psi \kappa)|)$ a polynomial in  the variables $|\psi(\theta) \kappa)|, \ldots,|\partial_{s}^{h} (\psi(\theta) \kappa)| $, with \emph{constants} coefficients  such that each monomial is of the form
\begin{align*}
C \prod_{l=0}^{h} |\partial_{s}^{l} (\psi\kappa)|^{\beta_{l}}
\qquad \text{ with } \sum_{l=0}^{h}(l+1)\beta_{l}=\sigma.
\end{align*}
We denote with $q_{\sigma}(\partial_{t}^{j} \lambda, \partial_{s}^{h}(\psi \kappa))$ a polynomial as before in $\lambda, \ldots, \partial_{t}^{j} \lambda$ and $(\psi(\theta) \kappa), \ldots,\partial_{s}^{h} (\psi(\theta) \kappa) $ with coefficient functions $C=C(\frac{1}{\psi}, \psi, \ldots, \partial_{\theta}^{h+1}\psi)$  such that all its monomial are of the form
\begin{align*}
C(\frac{1}{\psi}, \psi, \ldots, \partial_{\theta}^{h+1}\psi) \prod_{l=0}^{j} (\partial_{t}^{l} \lambda)^{\alpha_{l}} \prod_{l=0}^{h} (\partial_{s}^{l} (\psi\kappa))^{\beta_{l}}
\qquad \text{ with }\sum_{l=0}^{j}(2l+1)\alpha_{l}+ \sum_{l=0}^{h}(l+1)\beta_{l}=\sigma.
\end{align*}

We exemplify the notation just introduced in the next lemma (which is partially the anisotropic counterpart of \cite[Lemma 3.7]{MNT} and)  which will be used subsequently.
\begin{lemma}
\label{lem3.7}
 For $j=0,1,2$ we have that
\begin{align*}
\partial_{t} \partial_{s}^{j} (\psi \kappa)= \psi \partial_{s}^{j+2}(\psi \kappa) + \lambda \partial_{s}^{j+1} (\psi \kappa)
+ C(\frac{1}{\psi}, \psi') (\psi \kappa) \partial_{s}^{j+1}(\psi \kappa) 
+ p_{j+3} (\partial_{s}^{j } (\psi \kappa)).
\end{align*}
In particular it follows that
\begin{align*}
\theta_{tt}&= [(\psi \kappa)_{s} + \lambda \kappa]_{t} = \partial_{t} \partial_{s}(\psi \kappa) + (\lambda \kappa)_{t}\\
&= (\lambda \kappa)_{t} +
\psi (\psi \kappa)_{sss} + \lambda  (\psi \kappa)_{ss}
+ C(\frac{1}{\psi}, \psi') (\psi \kappa) (\psi \kappa)_{ss} 
+ p_{4} (\partial_{s} (\psi \kappa))\\
&= \psi (\psi \kappa)_{sss} + q_{4}(\lambda_{t}, \partial_{s}^{2} (\psi \kappa)).
\end{align*}
\end{lemma}
\begin{proof}
Using Lemma~\ref{lemma2.1} and \eqref{boundmpiccolo} we obtain (we write $\psi'$ for $\partial_{\theta} \psi$ and write $\psi$ instead of $\psi(\theta)$ to simplify the notation)
\begin{align}\label{oil}
\partial_{t} (\psi \kappa)&= \psi_{t} \kappa + \psi \kappa_{t}
=\psi' ((\psi \kappa)_{s} +\lambda \kappa) \kappa + \psi [ (\psi\kappa)_{ss} + \psi \kappa^{3} +\lambda \kappa_{s}] \notag \\
& = \psi (\psi \kappa)_{ss} +  \lambda (\psi \kappa)_{s} + \frac{\psi'}{\psi} (\psi \kappa)_{s} (\psi \kappa) + \frac{1}{\psi} (\psi \kappa)^{3} \\
& =\psi (\psi \kappa)_{ss} +  \lambda (\psi \kappa)_{s}  + C(\frac{1}{\psi}, \psi') (\psi \kappa)_{s} (\psi \kappa) + C (\frac{1}{\psi}
) (\psi \kappa)^{3} \notag
\end{align}
and the claim follows for $j=0$.
Next, using again Lemma~\ref{lemma2.1}, the previous step and \eqref{rule}, we compute
\begin{align*}
\partial_{t} (\psi \kappa)_{s} & = \partial_{t} \partial_{s}  (\psi \kappa)
= \partial_{s} \partial_{t} (\psi \kappa) + \psi \kappa^{2} (\psi \kappa)_{s} -\lambda_{s} (\psi \kappa)_{s} \\
& = \psi' \kappa (\psi \kappa)_{ss} + \psi (\psi \kappa)_{sss} + \lambda_{s} (\psi \kappa)_{s} + \lambda (\psi \kappa)_{ss} +\\
& \quad +\left( \frac{\psi''}{\psi} - \frac{(\psi')^{2} }{ \psi^{2}}\right)\frac{1}{\psi} (\psi\kappa) (\psi \kappa)_{s} (\psi \kappa) 
+ \frac{\psi'}{\psi}  ( (\psi \kappa)_{ss} (\psi \kappa) + (\psi \kappa)_{s} (\psi \kappa)_{s})\\
& \quad -\frac{\psi'}{\psi^{3}} (\psi \kappa) (\psi \kappa)^{3}  + \frac{3}{\psi}(\psi \kappa)^{2} (\psi \kappa)_{s}
+ \frac{1}{\psi}(\psi \kappa)^{2} (\psi \kappa)_{s} -\lambda_{s} (\psi \kappa)_{s} \\
& = \psi (\psi \kappa)_{sss}  + \lambda (\psi \kappa)_{ss} + C(\frac{1}{\psi}, \psi') (\psi \kappa)_{ss} (\psi \kappa) +
C(\frac{1}{\psi},\psi,  \psi', \psi'') (\psi\kappa)^{2} (\psi \kappa)_{s} \\
& \quad + C(\frac{1}{\psi}, \psi') ((\psi \kappa)_{s})^{2}
+ C(\frac{1}{\psi}, \psi') (\psi \kappa)^{4}.
\end{align*}
The case $j=2$ is computed analogously. 
The last statement follows by the definition of the polynomial $q_{4}(\lambda_{t}, \partial_{s}^{2} (\psi \kappa))$ and the fact that by Lemma~\ref{lemma2.1} we can write
$$ k_{t}= (\psi \kappa)_{ss} +\frac{1}{\psi^{2}} (\psi \kappa)^{3} + \lambda \partial_{s} \left( \frac{1}{\psi} (\psi \kappa) \right)
=(\psi \kappa)_{ss} +\frac{1}{\psi^{2}} (\psi \kappa)^{3} + \frac{\lambda}{\psi} (\psi \kappa)_{s} - \lambda\frac{\psi'}{\psi^{3}} (\psi\kappa)^{2}.
$$
\end{proof}

Next we  apply Lemma~\ref{structurelemma} to 
$$ S:= \psi (\psi \kappa)_{ss} \nu$$
which is a term in the normal component of $u_{tt}$ (see \eqref{uttb} below). We have that
\begin{align*}
S&= \psi (\psi \kappa)_{ss} \nu,\\
 |S|^{2} &=  \psi^{2} |(\psi \kappa)_{ss}|^{2},\\
S_{s}&= (\psi' \kappa(\psi \kappa)_{ss} + \psi (\psi \kappa)_{sss})\nu - (\psi \kappa) (\psi\kappa)_{ss} \tau\\
 & =\left(\frac{\psi'}{\psi} (\psi \kappa)(\psi \kappa)_{ss} + \psi (\psi \kappa)_{sss} \right)\nu - (\psi \kappa) (\psi\kappa)_{ss} \tau,\\
 S\cdot S_{s} &=  \psi' (\psi \kappa) ((\psi \kappa)_{ss})^{2} +  \psi^{2} (\psi \kappa)_{ss}(\psi \kappa)_{sss} ,\\
|S_{s}|^{2} & =\left(\frac{\psi'}{\psi} \right)^{2} |(\psi \kappa) (\psi \kappa)_{ss}|^{2} + \psi^{2} |(\psi \kappa)_{sss}|^{2} +
|(\psi \kappa) (\psi\kappa)_{ss}|^{2} 
+ 2\psi' (\psi \kappa)(\psi \kappa)_{ss}(\psi \kappa)_{sss}\\
&=  \psi^{2} |(\psi \kappa)_{sss}|^{2} + 2\psi' (\psi \kappa)(\psi \kappa)_{ss}(\psi \kappa)_{sss} +p_{8} (\partial_{s}^{2} (\psi \kappa)),\\
S_{ss}&= (\psi (\psi \kappa)_{ssss}+ 2\frac{\psi'}{\psi} (\psi \kappa) (\psi \kappa)_{sss}+ p_{5} (\partial_{s}^{2} (\psi \kappa)) )\nu + (\ldots) \tau.
\end{align*}
Moreover using Lemma~\ref{lem3.7} with $j=2$, and Lemma~\ref{lemma2.1} we can write
\begin{align*}
S_{t}&=[\psi_{t} (\psi \kappa)_{ss } +  \psi\partial_{t} \partial_{s}^{2}(\psi \kappa)] \nu + (\ldots) \tau\\
&=[\psi' ( (\psi\kappa)_{s}+ \lambda \kappa) (\psi \kappa)_{ss }  + \psi \big( \psi \partial_{s}^{4}(\psi \kappa) + \lambda \partial_{s}^{3} (\psi \kappa)\\
& \qquad + C(\frac{1}{\psi}, \psi') (\psi \kappa) \partial_{s}^{3}(\psi \kappa) 
 +p_{5} (\partial_{s}^{2} (\psi \kappa))  \, \big)] \nu + (\ldots) \tau\\
 & = \{ \psi^{2} (\psi\kappa)_{ssss} +\lambda [\psi' \kappa (\psi \kappa)_{ss} + \psi\partial_{s}^{3}(\psi \kappa) ] \\
 & \qquad +C(\frac{1}{\psi}, \psi') (\psi \kappa) \partial_{s}^{3}(\psi \kappa) 
 +p_{5} (\partial_{s}^{2} (\psi \kappa))    \}\nu + (\ldots) \tau.
\end{align*}
Therefore
\begin{align*}
(S_{t}-\psi S_{ss})^{\perp}&=\Big ( \lambda [\psi' \kappa (\psi \kappa)_{ss} + \psi \partial_{s}^{3}(\psi \kappa) ] +C(\frac{1}{\psi}, \psi') (\psi \kappa) \partial_{s}^{3}(\psi \kappa) \\
& \quad  +p_{5} (\partial_{s}^{2} (\psi \kappa))   -2 \psi' (\psi \kappa) (\psi \kappa)_{sss} \Big) \nu ,\\
 (S_{t}-\psi S_{ss}) \cdot S &= \lambda [\psi'  (\psi\kappa) ((\psi \kappa)_{ss})^{2} + \psi^{2} (\psi \kappa)_{ss}\partial_{s}^{3}(\psi \kappa) ]  \\
& \qquad  +C(\frac{1}{\psi},\psi, \psi') (\psi \kappa) (\psi \kappa)_{ss} \partial_{s}^{3}(\psi \kappa) 
 +p_{8} (\partial_{s}^{2} (\psi \kappa)),
\end{align*}
and we obtain
\begin{align*}
S \cdot (S_{t} - \psi S_{ss})& - \lambda (S \cdot S_{s}) - (S \cdot S_{s}) (\psi)_{s} -\frac{1}{2} |S|^{2} (\psi \kappa^{2})\\
&=C(\frac{1}{\psi},\psi, \psi') (\psi \kappa) (\psi \kappa)_{ss} \partial_{s}^{3}(\psi \kappa) 
 +p_{8} (\partial_{s}^{2} (\psi \kappa)),
\end{align*}
as well as
\begin{align*}
\frac{1}{2} |S|^{2} (\psi(\theta) \kappa)_{s}
 -\psi(\theta) \kappa (S \cdot S_{s}) = C(\frac{1}{\psi},\psi, \psi') (\psi \kappa) (\psi \kappa)_{ss} \partial_{s}^{3}(\psi \kappa) 
 +p_{8} (\partial_{s}^{2} (\psi \kappa)).
\end{align*}
Plugging the above expression into \eqref{structure} yields
\begin{align*}
\frac{d}{dt} &\left ( \frac{1}{2} \int_{I} |S|^{2} \frac{1}{\varphi^{\circ} (\nu)} ds \right) + \int_{I} \psi^{2} |(\psi \kappa)_{sss}|^{2}
\frac{\psi(\theta)}{\varphi^{\circ} (\nu)} ds = 
 \left[ (S \cdot S_{s}) \frac{\psi (\theta)}{\varphi^{\circ}(\nu) }+\frac{1}{2}|S|^{2}\frac{\lambda}{\varphi^{\circ}(\nu)}  \right]_{0}^{1}  \\
&+\int_{I} \left(C(\frac{1}{\psi},\psi, \psi') (\psi \kappa) (\psi \kappa)_{ss} \partial_{s}^{3}(\psi \kappa) 
 +p_{8} (\partial_{s}^{2} (\psi \kappa)) \right) \frac{1}{\varphi^{\circ} (\nu)}  ds \\ 
&  +\int_{I} \frac{D \varphi^{\circ} (\nu) \cdot \tau}{(\varphi^{\circ} (\nu))^{2}} \left(  C(\frac{1}{\psi},\psi, \psi') (\psi \kappa) (\psi \kappa)_{ss} \partial_{s}^{3}(\psi \kappa) 
 +p_{8} (\partial_{s}^{2} (\psi \kappa)) \right) ds .
\end{align*}
With help of Young inequality and using \eqref{boundmpiccolo} we achieve
\begin{align}\label{quasiquasi}
\frac{d}{dt} &\left ( \frac{1}{2} \int_{I} |S|^{2} \frac{1}{\varphi^{\circ} (\nu)} ds \right) + \frac{1}{2}\int_{I} \psi^{2} |(\psi \kappa)_{sss}|^{2}
\frac{\psi(\theta)}{\varphi^{\circ} (\nu)} ds \\
&\leq 
 \left[ (S \cdot S_{s}) \frac{\psi (\theta)}{\varphi^{\circ}(\nu) }+\frac{1}{2}|S|^{2}\frac{\lambda}{\varphi^{\circ}(\nu)}  \right]_{0}^{1}  
 + C \int_{I}  p_{8} (|\partial_{s}^{2} (\psi \kappa)|) ds.  \notag
\end{align}
To treat the boundary term 
\begin{align*}
&\left[ (S \cdot S_{s}) \frac{\psi (\theta)}{\varphi^{\circ}(\nu) }+\frac{1}{2}|S|^{2}\frac{\lambda}{\varphi^{\circ}(\nu)}  \right]_{0}^{1} \\
&= \left[  \frac{\psi}{\varphi^{\circ}(\nu)} \Big ( \psi' (\psi \kappa) ((\psi \kappa)_{ss})^{2} +  \psi^{2} (\psi \kappa)_{ss}(\psi \kappa)_{sss} \Big)  + \frac{\lambda}{2\varphi^{\circ}(\nu)} \psi^{2} |(\psi \kappa)_{ss}|^{2} \right]_{0}^{1} 
\end{align*}
it is imperative to be able to  lower the order of the term with three spacial derivatives.
Note that the $\lambda$-term is of type
\begin{align*}
\frac{\lambda}{2\varphi^{\circ}(\nu)} \psi^{2} |(\psi \kappa)_{ss}|^{2} = \frac{1}{\varphi^{\circ}(\nu)} q_{7}(\partial_{t} \lambda, \partial_{s}^{2} (\psi \kappa) ).
\end{align*}
To handle the term $(S \cdot S_{s})$ observe that by Lemma~\ref{lemma2.1} and \eqref{oil} we can write
\begin{align} \label{uttb}
u_{tt} & =[ \psi \kappa \nu + \lambda \tau]_{t} = (\psi \kappa)_{t} \nu +  (\psi \kappa) \nu_{t} + \lambda_{t} \tau + \lambda \tau_{t} \\ 
&=  [ (\psi \kappa)_{t} +\lambda \theta_{t}] \nu + (\lambda_{t} - (\psi \kappa) \theta_{t}) \tau \notag\\
&= [\psi (\psi \kappa)_{ss} +  \lambda (\psi \kappa)_{s} + \frac{\psi'}{\psi} (\psi \kappa)_{s} (\psi \kappa) + \frac{1}{\psi} (\psi \kappa)^{3} + \lambda \theta_{t}] \nu + (\lambda_{t} - (\psi \kappa) \theta_{t}) \tau \notag\\
& = S +  [\lambda (\psi \kappa)_{s} + \frac{\psi'}{\psi} (\psi \kappa)_{s} (\psi \kappa) + \frac{1}{\psi} (\psi \kappa)^{3} + \lambda \theta_{t}] \nu  + (\ldots) \tau. \notag
\end{align}
At the fixed boundary point (that is at $x=1$) we have that $(\psi\kappa)= \lambda=\lambda_{t}=(\psi\kappa)_{ss}=0$ since the here
 $u_{t}=u_{tt}=0$. Hence we need to treat only the boundary terms at the junction point.
Here we have, using \eqref{uttb},
\begin{align*}
(S \cdot S_{s}) & = S \cdot (S_{s})^{\perp} = (u_{tt} \cdot (S_{s})^{\perp})\\
& \quad  -
[\lambda (\psi \kappa)_{s} + \frac{\psi'}{\psi} (\psi \kappa)_{s} (\psi \kappa) + \frac{1}{\psi} (\psi \kappa)^{3} + \lambda\theta_{t}]  
(\psi' \kappa(\psi \kappa)_{ss} + \psi (\psi \kappa)_{sss})\\
& =(u_{tt} \cdot (S_{s})^{\perp}) -R.
\end{align*}
Concerning the term $R$ we observe
\begin{align*}
&[\lambda (\psi \kappa)_{s} + \frac{\psi'}{\psi} (\psi \kappa)_{s} (\psi \kappa) + \frac{1}{\psi} (\psi \kappa)^{3} + \lambda\theta_{t}]  
(\psi' \kappa(\psi \kappa)_{ss})\\
&=[2\lambda (\psi \kappa)_{s} + \frac{\psi'}{\psi} (\psi \kappa)_{s} (\psi \kappa) + \frac{1}{\psi} (\psi \kappa)^{3} + \lambda^{2} \frac{1}{\psi}(\psi \kappa) ]  
(\frac{\psi'}{\psi} ( \psi\kappa)(\psi \kappa)_{ss}) =  q_{7}(\partial_{t} \lambda, \partial_{s}^{2} (\psi \kappa) ).
\end{align*}
Using Lemma~\ref{lem3.7} we also compute
\begin{align*}
&[\lambda (\psi \kappa)_{s} + \frac{\psi'}{\psi} (\psi \kappa)_{s} (\psi \kappa) + \frac{1}{\psi} (\psi \kappa)^{3} + \lambda\theta_{t}]  
\psi (\psi \kappa)_{sss}\\
&=[( 2\lambda + \psi'\kappa )  (\psi \kappa)_{s}  + \frac{1}{\psi} (\psi \kappa)^{3} + \lambda^{2}  \kappa ] (\partial_{t} (\psi \kappa)_{s} \,  + q_{4} (\partial_{t} \lambda, \partial_{s}^{2} (\psi \kappa)))\\
&= ( 2\lambda + \psi'\kappa ) \left( \frac{|(\psi \kappa)_{s} |^{2}}{2} \right)_{t} + (\frac{1}{\psi} (\psi \kappa)^{3} + \lambda^{2}  \kappa ) (\partial_{t} (\psi \kappa)_{s})   +q_{7}(\partial_{t} \lambda, \partial_{s}^{2} (\psi \kappa) )\\
& = \left ( ( 2\lambda + \frac{\psi'}{\psi}(\psi\kappa) )  \frac{|(\psi \kappa)_{s} |^{2}}{2}  + (\frac{1}{\psi} (\psi \kappa)^{3} + \frac{1}{\psi}\lambda^{2} (\psi \kappa)  )(\psi \kappa)_{s} \right)_{t} \\
& \quad -  ( 2\lambda + \psi'\kappa )_{t}  \frac{|(\psi \kappa)_{s} |^{2}}{2} - (\psi \kappa)_{s} (\frac{1}{\psi} (\psi \kappa)^{3} + \lambda^{2}  \kappa)_{t} + q_{7}(\partial_{t} \lambda, \partial_{s}^{2} (\psi \kappa) )\\
&=\partial_{t} \big( q_{5}(\lambda, \partial_{s}(\psi \kappa)) \big)  +q_{7}(\partial_{t} \lambda, \partial_{s}^{2} (\psi \kappa) ).
\end{align*}
Hence
$$R= \partial_{t} \big( q_{5}(\lambda, \partial_{s}(\psi \kappa)) \big)  +q_{7}(\partial_{t} \lambda, \partial_{s}^{2} (\psi \kappa) ).$$
Therefore we obtain that at the junction point we have
\begin{align*}
\frac{\psi (\theta)}{\varphi^{\circ}(\nu)}(S \cdot S_{s})  &= \frac{\psi (\theta)}{\varphi^{\circ}(\nu)}  (u_{tt} \cdot (S_{s})^{\perp}) - \frac{\psi (\theta)}{\varphi^{\circ}(\nu)} \left(\partial_{t} \big( q_{5}(\lambda, \partial_{s}(\psi \kappa)) \big)  +q_{7}(\partial_{t} \lambda, \partial_{s}^{2} (\psi \kappa) ) \right) \\
&=\frac{\psi (\theta)}{\varphi^{\circ}(\nu)}  (u_{tt} \cdot (S_{s})^{\perp}) - \frac{1}{\varphi^{\circ}(\nu)} \left[\partial_{t} \big( q_{5}(\lambda, \partial_{s}(\psi \kappa)) \big)  +q_{7}(\partial_{t} \lambda, \partial_{s}^{2} (\psi \kappa) ) \right] .
\end{align*}
Next, using \eqref{uttb} and the expression derived above for $S_{s}$, we observe that 
\begin{align*}
\frac{\psi (\theta)}{\varphi^{\circ}(\nu)} & (u_{tt} \cdot (S_{s})^{\perp})=\\
&= \frac{\psi (\theta)}{\varphi^{\circ}(\nu)}[ (\psi \kappa)_{t} +\lambda \theta_{t}] (\frac{\psi'}{\psi} (\psi\kappa)(\psi \kappa)_{ss} + \psi (\psi \kappa)_{sss})\\
& =\frac{\psi (\theta)}{\varphi^{\circ}(\nu)}[ (\psi \kappa)_{t} +\lambda \theta_{t}]\psi (\psi \kappa)_{sss} + \frac{1}{\varphi^{\circ}(\nu)} q_{7}(\partial_{t} \lambda, \partial_{s}^{2} (\psi \kappa) )\\
& =\frac{\psi (\theta)}{\varphi^{\circ}(\nu)}[ (\psi \kappa)_{t} +\lambda \theta_{t}](\theta_{tt} + q_{4}(\partial_{t} \lambda, \partial_{s}^{2} (\psi \kappa) ))+ \frac{1}{\varphi^{\circ}(\nu)} q_{7}(\partial_{t} \lambda, \partial_{s}^{2} (\psi \kappa) )\\
&= \frac{1}{\varphi^{\circ}(\nu)}  q_{7}(\partial_{t} \lambda, \partial_{s}^{2} (\psi \kappa) )+ \frac{\psi (\theta)}{\varphi^{\circ}(\nu)}[ (\psi \kappa)_{t} +\lambda \theta_{t}]\theta_{tt}
\end{align*}
where we have used Lemma~\ref{lem3.7} in the second last equality.
Hence so far we have shown that 
\begin{align}\label{qqquasi}
&\left[ (S \cdot S_{s}) \frac{\psi (\theta)}{\varphi^{\circ}(\nu) }+\frac{1}{2}|S|^{2}\frac{\lambda}{\varphi^{\circ}(\nu)}  \right]_{0}^{1} = - \left( (S \cdot S_{s}) \frac{\psi (\theta)}{\varphi^{\circ}(\nu) }+\frac{1}{2}|S|^{2}\frac{\lambda}{\varphi^{\circ}(\nu)}  \right)\Big|_{x=0} \\
&=\frac{1}{\varphi^{\circ}(\nu)}q_{7}(\partial_{t} \lambda, \partial_{s}^{2} (\psi \kappa) ) 
+ \frac{1}{\varphi^{\circ}(\nu)} 
\partial_{t} \big( q_{5}(\lambda, \partial_{s}(\psi \kappa)) \big)  
- \frac{\psi (\theta)}{\varphi^{\circ}(\nu)}[ (\psi \kappa)_{t} +\lambda \theta_{t}]\theta_{tt}.\notag
\end{align}
To handle the last term we use the boundary conditions: twice derivation in time of \eqref{HC} gives (at the junction point)
\begin{align*}
0 &=-\sum_{i=1}^{3} (D \varphi^{\circ}(\nu^{i}))_{tt} = \sum_{i=1}^{3}  (D^{2}\varphi^{\circ}(\nu^{i}) \theta^{i}_{t} \tau^{i})_{t}
=\sum_{i=1}^{3} D^{3}\varphi^{\circ}(\nu^{i}) \nu^{i}_{t}\theta^{i}_{t} \tau^{i} + \sum_{i=1}^{3}  D^{2}\varphi^{\circ}(\nu^{i}) \theta^{i}_{tt} \tau^{i} \\
&= - \sum_{i=1}^{3} D^{3}\varphi^{\circ}(\nu^{i}) \tau^{i}\tau^{i} (\theta^{i}_{t} )^{2} + \sum_{i=1}^{3}  \frac{\psi (\theta^{i})}{\varphi^{\circ} (\nu^{i})}
\theta^{i}_{tt} \tau^{i}.
\end{align*}
Since here $u_{tt}^{1}=u_{tt}^{2}=u_{tt}^{3}$, we obtain  $Ru_{tt}^{1}=Ru_{tt}^{2}=Ru_{tt}^{3}$ with $R=\tiny{\left(\begin{array}{cc} 0&-1\\1&0 \end{array} \right)}$ which rotates vectors by $\pi/2$,  and hence (recall \eqref{uttb})
\begin{align*}
0&=Ru_{tt}^{1} \cdot (- \sum_{i=1}^{3} D^{3}\varphi^{\circ}(\nu^{i}) \tau^{i}\tau^{i} (\theta^{i}_{t} )^{2} + \sum_{i=1}^{3}  \frac{\psi (\theta^{i})}{\varphi^{\circ} (\nu^{i})}
\theta^{i}_{tt} \tau^{i})\\
&=\sum_{i=1}^{3} (R u_{tt}^{i}) \cdot (-D^{3}\varphi^{\circ}(\nu^{i}) \tau^{i}\tau^{i} (\theta^{i}_{t} )^{2} +   \frac{\psi (\theta^{i})}{\varphi^{\circ} (\nu^{i})}
\theta^{i}_{tt} \tau^{i})\\
&=\sum_{i=1}^{3} (- [ (\psi(\theta^{i}) \kappa^{i})_{t} +\lambda^{i} \theta^{i}_{t}] \tau^{i} + (\lambda^{i}_{t} - (\psi(\theta^{i}) \kappa^{i}) \theta^{i}_{t}) \nu^{i} )
\cdot (-D^{3}\varphi^{\circ}(\nu^{i}) \tau^{i}\tau^{i} (\theta^{i}_{t} )^{2} +   \frac{\psi (\theta^{i})}{\varphi^{\circ} (\nu^{i})}\theta^{i}_{tt} \tau^{i}).
\end{align*}
It follows that
\begin{align*} &\sum_{i=1}^{3} \frac{\psi (\theta^{i})}{\varphi^{\circ} (\nu^{i})}\theta^{i}_{tt} [ (\psi(\theta^{i}) \kappa^{i})_{t} +\lambda^{i} \theta^{i}_{t}] \\
&= \sum_{i=1}^{3}
D^{3}\varphi^{\circ}(\nu^{i}) \tau^{i}\tau^{i} \tau^{i} (\theta^{i}_{t} )^{2} [ (\psi(\theta^{i}) \kappa^{i})_{t} +\lambda^{i} \theta^{i}_{t}] -\sum_{i=1}^{3}
D^{3}\varphi^{\circ}(\nu^{i}) \tau^{i}\tau^{i} \nu^{i} (\theta^{i}_{t} )^{2} (\lambda^{i}_{t} - (\psi(\theta^{i}) \kappa^{i}) \theta^{i}_{t}) \\
&= \sum_{i=1}^{3}  (D^{3}\varphi^{\circ}(\nu^{i}) \tau^{i}\tau^{i} \tau^{i} + D^{3}\varphi^{\circ}(\nu^{i}) \tau^{i}\tau^{i} \nu^{i} ) \, \, q_{7}(\partial_{t} \lambda^{i}, \partial_{s}^{2} (\psi(\theta^{i}) \kappa^{i}) ),
\end{align*}
where note that $|D^{3}\varphi^{\circ}(\nu^{i}) \tau^{i}\tau^{i} \tau^{i} + D^{3}\varphi^{\circ}(\nu^{i})\tau^{i}\tau^{i} \nu^{i}| \leq C$.
The expression above together with \eqref{quasiquasi} and  \eqref{qqquasi} yields
\begin{align}\label{uffa2}
\sum_{i=1}^{3}\frac{d}{dt} &\left ( \frac{1}{2} \int_{I} |\psi(\theta^{i}) (\psi(\theta^{i}) \kappa^{i})_{ss}|^{2} \frac{1}{\varphi^{\circ} (\nu^{i})} ds \right) + \frac{1}{2}\int_{I} (\psi (\theta^{i}))^{2} |(\psi(\theta^{i}) \kappa^{i})_{sss}|^{2}
\frac{\psi(\theta^{i})}{\varphi^{\circ} (\nu^{i})} ds \\
&\leq 
\sum_{i=1}^{3} 
\partial_{t} \left( \frac{1}{\varphi^{\circ} (\nu^{i})}  q_{5}(\lambda^{i}, \partial_{s}(\psi(\theta^{i}) \kappa^{i})) \right)\Big|_{x=0}  + \sum_{i=1}^{3}  C|q_{7}(\partial_{t} \lambda^{i}, \partial_{s}^{2} (\psi(\theta^{i}) \kappa^{i}) )| \Big|_{x=0} \notag \\
&\quad + \sum_{i=1}^{3}  C\int_{I}  p_{8} (|\partial_{s}^{2} (\psi(\theta^{i}) \kappa^{i})|) ds. \notag
\end{align}
Finally we apply interpolation inequalities.  Using Proposition~\ref{IEst} and H\"older inequality as demonstrated and carefully explained in \cite[p.260-261]{MNT}  we obtain that
\begin{align}\label{p8est}
\int_{I}  p_{8} (|\partial_{s}^{2} (\psi(\theta^{i}) \kappa^{i})|) ds \leq  \epsilon 
\int_{I} (\psi (\theta^{i}))^{2} |(\psi(\theta^{i}) \kappa^{i})_{sss}|^{2}
\frac{\psi(\theta^{i})}{\varphi^{\circ} (\nu^{i})} ds^{2}  + C_{\epsilon} \| \psi(\theta^{i}) \kappa^{i}\|_{L^{2}(I)}^{14} +C
\end{align}
where the constants  depends on \eqref{boundmpiccolo}, the anisotropy map, and  the bounds of the lengths of the curves.
 
At the triple junction recall that we can write $\lambda^{i}$ in terms of $(\psi(\theta^{j})\kappa^{j})$ for $j\neq i$.
In particular, we have  that \eqref{boundlbdry} holds. Together with \eqref{derla}, Lemma~\ref{lemma2.1} 
and Lemma~\ref{lem3.7} we infer that
\begin{lemma}\label{lemqp}
We have that at the junction point there holds
\begin{align*}
\sum_{i=1}^{3}  |q_{7}(\partial_{t} \lambda^{i}, \partial_{s}^{2} (\psi(\theta^{i}) \kappa^{i}) )|  \leq  C p_{7} (|\partial_{s}^{2} (\psi(\theta^{j}) \kappa^{j})|; j=1,2,3)\,,\\
\sum_{i=1}^{3} | q_{5}( \lambda^{i}, \partial_{s} (\psi(\theta^{i}) \kappa^{i}) )|  \leq  C p_{5} (|\partial_{s} (\psi(\theta^{j}) \kappa^{j})|; j=1,2,3)\,,
\end{align*}
where $C$ depends on the anisotropy map and where the polynomials on the right-hand side now contains derivatives of $(\psi(\theta^{j}) \kappa^{j})$ for the three different curves.
\end{lemma}
Using Lemma~\ref{lemqp}, interpolation estimates, and H\"older inequality as in \cite[p. 262]{MNT}  we obtain 
\begin{align}\label{q7est}
 \sum_{i=1}^{3}  & |q_{7}(\partial_{t} \lambda^{i}, \partial_{s}^{2} (\psi(\theta^{i}) \kappa^{i}) )|  
 \leq  C p_{7} (|\partial_{s}^{2} (\psi(\theta^{j}) \kappa^{j})|; j=1,2,3) \\
& \leq
 \sum_{i=1}^{3}  \epsilon 
\int_{I} (\psi (\theta^{i}))^{2} |(\psi(\theta^{i}) \kappa^{i})_{sss}|^{2}
\frac{\psi(\theta^{i})}{\varphi^{\circ} (\nu^{i})} ds^{2}  + C_{\epsilon} \| \psi(\theta^{i}) \kappa^{i}\|_{L^{2}(I)}^{14} +C.\notag
\end{align}
From \eqref{uffa2}, \eqref{p8est}, \eqref{q7est},
choosing $\epsilon$ appropriately, integrating in time and using  \eqref{boundmpiccolo} we obtain 
\begin{align*}
&\sum_{i=1}^{3}  \| (\psi(\theta^{i}) \kappa^{i})_{ss} \|_{L^{2}(I)}^{2}(t)  \leq 
C\sum_{i=1}^{3}  \| (\psi(\theta^{i}) \kappa^{i})_{ss} \|_{L^{2}(I)}^{2}(0)
+  Ct + \sum_{i=1}^{3} C\int_{0}^{t} \| \psi(\theta^{i}) \kappa^{i}\|_{L^{2}(I)}^{14} dt 
\\ & \quad +\sum_{i=1}^{3}
C | q_{5}(\lambda^{i}, \partial_{s}(\psi(\theta^{i}) \kappa^{i}))|(t)  \Big|_{(x=0)}  +\sum_{i=1}^{3}
C | q_{5}(\lambda^{i}, \partial_{s}(\psi(\theta^{i}) \kappa^{i}))|(0)  \Big|_{(x=0)} .
\end{align*}
By  Lemma~\ref{lemqp}, and together again with interpolation  and H\"older  inequalities (cp. with \cite[p263]{MNT}) we obtain 
that at the junction point we have, for any time $t$,
\begin{align*}
\sum_{i=1}^{3}
C | q_{5}(\lambda^{i}, \partial_{s}(\psi(\theta^{i}) \kappa^{i}))|(t) 
&\leq  C p_{5} (|\partial_{s} (\psi(\theta^{j}) \kappa^{j})|; j=1,2,3)\\
& \leq \frac{1}{2}\sum_{i=1}^{3}  \| (\psi(\theta^{i}) \kappa^{i})_{ss} \|_{L^{2}(I)}^{2}(t)  +C \| \psi(\theta^{i}) \kappa^{i}\|_{L^{2}(I)}^{10}(t),
\end{align*}
so that we  finally infer
\begin{align*}
\sum_{i=1}^{3}  \| (\psi(\theta^{i}) \kappa^{i})_{ss} \|_{L^{2}(I)}^{2}(t)  \leq  C_{0} + Ct + \sum_{i=1}^{3} C \left(\| \psi(\theta^{i}) \kappa^{i}\|_{L^{2}(I)}^{10}(t)  + \int_{0}^{t} \|  \psi(\theta^{i}) \kappa^{i}\|_{L^{2}(I)}^{14} dt   \right),
\end{align*}
where 
$$C_{0}= C\sum_{i=1}^{3}  \| (\psi(\theta^{i}) \kappa^{i})_{ss} \|_{L^{2}(I)}^{2}(0) + C \sum_{i=1}^{3}\| \psi(\theta^{i}) \kappa^{i}\|_{L^{2}(I)}^{10} (0)$$
and $C$ depends on \eqref{boundmpiccolo}, the anisotropy map, and  the  bound on the lengths of the curves.

Upon recalling that $\kappa_{\varphi}=\frac{1}{\varphi^{\circ}(\nu)} (\psi (\theta)\kappa)$ and interpolation inequalities
from Proposition~\ref{IEst}  we can summarize our above findings as follows:

\begin{lemma}\label{lemstima}
If for $0< T <\infty$, the lengths of the curves of the network are uniformly bounded from below
\begin{align*}
L(u^{i}(t)) \geq \delta >0, \qquad i=1,2,3, \text{ for any } t \in [0,T),
\end{align*}
and we have a uniform bound for the curvatures
\begin{align*}
\sup_{ t \in [0,T)}  \sum_{i=1}^{3}\| \kappa_{\varphi}^{i}\|_{L^{2}(I)} \leq C_{K}
\end{align*}
then 
\begin{align*}
\sup_{ t \in [0,T)} \sum_{i=1}^{3} (\| (\kappa_{\varphi}^{i})_{s}\|_{L^{2}(I)}+ \| (\kappa_{\varphi}^{i})_{ss}\|_{L^{2}(I)}) \leq C\\
\sup_{ t \in [0,T)} \sum_{i=1}^{3} (\| \kappa^{i}_{s}\|_{L^{2}(I)}+ \| \kappa^{i}_{ss}\|_{L^{2}(I)}) \leq C,
\end{align*}
hold for a solution of the geometric problem (cf. Section~\ref{sec:geoprob}).
The constant  $C$ depends on $\delta$, $C_{K}$, $T$, the initial data $\|(\psi\kappa^{i})_{ss}\|_{L^{2}}(0)$ for $i=1,2,3$, $m$, $M$ (recall \eqref{boundmpiccolo}), $a_{0}$ (recall \eqref{a0}), and
 on $ C^{-1} \leq \varphi^{\circ} (\nu)\leq C$,  $|D\varphi^{\circ} (\nu)|\leq C$,  $\sup_{S^{1}}(| \psi'| + |\psi''|+ |\psi'''|)$ . 
\end{lemma}

\subsection{Main result}
From Lemma \ref{lemstima}, Theorem \ref{teo:GEO} and Proposition \ref{pro:GEO} we finally obtain our main result 
on the behavior of a geometric solution at the maximal existence time.

\begin{theo}\label{teo:sciop}
Let $\alpha \in (0,1)$, $\sigma^{i}$ be as in Definition~\ref{admtriod}, and
$u^{i} \in C^{\frac{2+\alpha}{2}, 2+\alpha}([0,T) \times [0,1], \R^{2}) \cap C^{\infty} ((0,T) \times [0,1], \R^{2})$, 
$i=1,2,3$, be geometric solutions (as in Theorem~\ref{teo:GEO}) defined in the maximal time interval $[0,T)$.
Then we have
\begin{equation}\label{eqmaxmin}
\liminf_{t\to T} \min_{i\in\{1,2,3\}} L(u^i(t)) = 0
\qquad\text{or}\qquad
\limsup_{t\to T} \max_{i\in\{1,2,3\}} \| \kappa^i_\varphi\|_{L^2(I)} = +\infty.
\end{equation}
\end{theo}


\renewcommand{\thesection}{}

\appendix\renewcommand{\thesection}{\Alph{section}}
\setcounter{equation}{0}
\renewcommand{\theequation}{\Alph{section}\arabic{equation}}


\section{Some useful results }\label{AppA}
The following remark and the next three lemmas are a straight forward adaptation to the present setting of the lemmas presented in \cite[Appendix~B]{DLP-STE}.
\begin{rem}\label{B1}
If $v \in C^{ \frac{k+\alpha}{2}, k+\alpha}([0,T]\times[0,1])$, $k \in \mathbb{N}_0$, then $\partial_x^l v \in C^{ \frac{k-l+\alpha}{2}, k-l+\alpha}([0,T]\times[0,1])$ for all $0 \leq l\leq k$ and
$$\|\partial_x^l v \|_{C^{ \frac{k-l+\alpha}{2}, k-l+\alpha}([0,T]\times[0,1])} \leq \| v\|_{C^{ \frac{k+\alpha}{2}, k+\alpha}([0,T]\times[0,1])} \, .$$
In particular at each fixed $x \in [0,1]$ we have $\partial_x^l v (\cdot,x )\in C^{s,\beta}([0,T])$ 
with $s =[\frac{k-l+\alpha}{2}]$ and $\beta = \frac{k-l+\alpha}{2}-s$.
\end{rem}

\begin{lemma}\label{B2}
For $k \in \mathbb{N}_0$, $\alpha,\beta \in (0,1)$ and $T>0$ we have
\begin{enumerate}
\item if $v,w \in C^{\frac{k+\alpha}{2},k+\alpha}([0,T]\times [0,1])$, then
$$ \| v w \|_{C^{\frac{k+\alpha}{2},k+\alpha}} \leq C \| v \|_{C^{\frac{k+\alpha}{2},k+\alpha}} \| w \|_{C^{\frac{k+\alpha}{2},k+\alpha}} \, ,$$
with $C=C(k)>0$;
\item if $v\in C^{\frac{\alpha}{2},\alpha}([0,T]\times [0,1])$, $v(t,x)\ne 0$ for all $(t,x)$, then
$$ \Big\| \frac{1}{v} \Big\|_{C^{\frac{\alpha}{2},\alpha}} \leq \Big\| \frac{1}{v} \Big\|^2_{C^{0}([0,T]\times [0,1])} \| v \|_{C^{\frac{\alpha}{2},\alpha}} \, .$$ 
\end{enumerate}
Similar statements are true for functions in $C^{k,\beta}([0,T])$ and $C^{k,\beta}([0,1])$.
\end{lemma}

\begin{lemma}\label{B3}
For $n \in \mathbb{N}$, $k \in \mathbb{N}_0$, $\alpha,\beta \in (0,1)$ and $T>0$ we have
\begin{enumerate}
\item 
if a vector-field $v\in C^{\frac{\alpha}{2},\alpha}([0,T]\times [0,1];\R^n)$, then
$$ \|\, |v|\, \|_{C^{\frac{\alpha}{2},\alpha}} \leq C  \| v \|_{C^{\frac{\alpha}{2},\alpha}} \, ,$$
with $C=C(n)$.
\item for $v, w\in C^{\frac{\alpha}{2},\alpha}([0,T]\times [0,1];\R^n)$  we have
$$ \|\, |v|- |w|\, \|_{C^{\frac{\alpha}{2},\alpha}} \leq  C  \left \| \frac{1}{|v|+|w|} \right\|^{2}_{C^{0} ([0,T]\times [0,1])}
 (\|  v \|_{C^{\frac{\alpha}{2},\alpha}} + \|  w\|_{C^{\frac{\alpha}{2},\alpha}} )^{2} \| v-w \|_{C^{\frac{\alpha}{2},\alpha}} $$
\end{enumerate}
with $C=C(n)$.
Similar statements are true for functions in $C^{k,\beta}([0,T])$ and $C^{k,\beta}([0,1])$.
\end{lemma}

\begin{lemma}\label{B5}
Let $T<1$ and $v \in C^{ \frac{2+\alpha}{2}, 2+\alpha}([0,T]\times[0,1])$ such that $v(0,x)= 0$, for any $x \in [0,1]$ then
$$ \| \partial_x^l v \|_{C^{\frac{m+\alpha}{2}, m+\alpha}} \leq C(m) T^{\beta} \|  v \|_{C^{\frac{2+\alpha}{2}, 2+\alpha}} $$
for all $l,m \in \mathbb{N}_0$ such that $l+m<2$. Here $\beta=\max \{ \frac{1-\alpha}{2}, \frac{\alpha}{2} \} \in (0,1)$; more precisely for $l= 1$ then $\beta=\frac{\alpha}{2}$.

In particular, for each $x \in [0,1]$ fixed
$$ \| \partial_x^l v (\cdot, x) \|_{C^{0,\frac{m+\alpha}{2}}([0,T])} \leq C(m) T^{\beta} \|  v \|_{C^{\frac{2+\alpha}{2}, 2+\alpha}} $$
for all $l,m \in \mathbb{N}_0$ such that $l+m<2$.
\end{lemma}

Next we provide a list of results that are useful in the contraction argument in the proof of the short-time existence.
In the following lemma we use  that, given $\sigma^{i} \in C^{2,\alpha}([0,1])$, then  $\sigma^{i} \in C^{\frac{2+\alpha}{2},2+\alpha}([0,T]\times [0,1])$ by extending it as a constant function in time. For the definition of $X_{i}$, $\delta$ and $T$ recall \eqref{defXi} and \eqref{eq:immersion} and the remarks in between.

\begin{lemma}\label{lem:hilfsatz}
Let $\sigma^{i} \in C^{2,\alpha}([0,1])$ and $\bar{u}^{i}, \bar{v}^{i} \in X_i$.
Then we have that 
\begin{equation*}
\|  \sigma^{i}_{x}  - \bar{u}^{i}_{x}\|_{C^{\frac{\alpha}{2},\alpha}([0,T]\times[0,1])} \leq C T^{\frac{\alpha}{2}}
\Big(\|\bar{u}^{i} \|_{C^{\frac{2+\alpha}{2},2+\alpha}([0,T]\times[0,1])} + \|\sigma^{i} \|_{C^{2,\alpha}([0,1])}\Big)
\end{equation*}
for some universal constant $C$. Moreover, for $T<1$ we have that
\begin{equation*}\label{eq:hol1}
\| | \sigma^{i}_{x} |- |\bar{u}^{i}_{x}|\|_{C^{\frac{\alpha}{2},\alpha}([0,T]\times[0,1])} \leq C T^{\frac{\alpha}{2}} \Big(\|\bar{u}^{i} \|_{C^{\frac{2+\alpha}{2},2+\alpha}([0,T]\times[0,1])} + \|\sigma^{i} \|_{C^{2,\alpha}([0,1])}\Big)^3\,  ,
\end{equation*}
with $C= C(\delta)$. Furthermore, for $m \in \mathbb{N}$ 
 we have
\allowdisplaybreaks{\begin{align*}
\Big\| \frac{1}{|\sigma^{i}_{x} |^m}- \frac{1}{| \bar{u}^{i}_{x}|^m} \Big\|_{C^{\frac{\alpha}{2},\alpha}([0,T]\times [0,1])} & \leq C   T^{\frac{\alpha}{2}} 
\\
\Big\| \frac{1}{| \sigma^{i}_{x} (\cdot, x)|^m}- \frac{1}{| \bar{u}^{i}_{x} (\cdot, x)|^m}\Big\|_{C^{0,\frac{1+\alpha}{2}}([0,T])} & \leq C T^{\frac{\alpha}{2}} 
\end{align*}
for any $x \in [0,1]$ and with  $C=C(m,\delta, \|\bar{u}^{i} \|_{C^{\frac{2+\alpha}{2},2+\alpha}([0,T]\times [0,1])}, \|\sigma^{i} \|_{C^{2,\alpha}([0,1])})$
as well as 
\begin{align*}
\Big\| \frac{1}{| \bar{u}^{i}_{x} |^m}- \frac{1}{| \bar{v}^{i}_{x}|^m}\Big\|_{C^{\frac{\alpha}{2},\alpha}([0,T]\times [0,1])} & \leq  C  T^{\frac{\alpha}{2}}  \| \bar{u}^{i} - \bar{v}^{i}\|_{C^{\frac{2+\alpha}{2},2+\alpha}([0,T]\times [0,1])}\, ,\\ 
\Big\| \frac{1}{| \bar{u}^{i}_{x} (\cdot, x)|^m}- \frac{1}{| \bar{v}^{i}_{x} (\cdot, x)|^m}\Big\|_{C^{0,\frac{1+\alpha}{2}}([0,T])} &\leq C  T^{\frac{\alpha}{2}}  \| \bar{u}^{i} - \bar{v}^{i}\|_{C^{\frac{2+\alpha}{2},2+\alpha}([0,T]\times [0,1])}\,, 
\end{align*}}
again for  $x \in [0,1]$ and  with $C=C(m,\delta, \|\bar{u}^{i} \|_{C^{\frac{2+\alpha}{2},2+\alpha}([0,T]\times [0,1])}, \|\bar{v}^{i} \|_{C^{\frac{2+\alpha}{2},2+\alpha}([0,T]\times [0,1])})$. 
\end{lemma}
\begin{proof}It follows by an adaptation to the present setting of \cite[Lemma~3.1]{DLP-STE} and \cite[Lemma~3.4]{DLP-STE} using Remark~\ref{B1} and  the Lemmas~\ref{B2}, \ref{B3}, \ref{B5} stated above.
\end{proof}
 
 \begin{lemma}\label{hilfsatz2}
 Let $\sigma^{i} \in C^{2,\alpha}([0,1])$, $\bar{u}^{i}, \bar{v}^{i} \in X_i$, $T<1$ and $x \in [0,1]$. Then we have
 \begin{align*}
 \left \|  \frac{\sigma^{i}_{x}}{|\sigma^{i}_{x}|}  -\frac{\bar{u}^{i}_{x}}{|\bar{u}^{i}_{x}|} \right \|_{C^{\frac{\alpha}{2},\alpha}([0,T]\times[0,1])} &\leq C T^{\frac{\alpha}{2}}, \\
 \left \|  \frac{\sigma^{i}_{x}}{|\sigma^{i}_{x}|}(\cdot, x)  -\frac{\bar{u}^{i}_{x}}{|\bar{u}^{i}_{x}|}(\cdot, x) \right \|_{C^{0,\frac{1+\alpha}{2}}([0,T])} &\leq C T^{\frac{\alpha}{2}}
 \end{align*}
 with $C=C(\delta, \|\bar{u}^{i} \|_{C^{\frac{2+\alpha}{2},2+\alpha}([0,T]\times [0,1])}, \|\sigma^{i} \|_{C^{2,\alpha}([0,1])})$. Similarly
 \begin{align*}
 \left \|  \frac{\bar{v}^{i}_{x}}{|\bar{v}^{i}_{x}|}  -\frac{\bar{u}^{i}_{x}}{|\bar{u}^{i}_{x}|} \right \|_{C^{\frac{\alpha}{2},\alpha}([0,T]\times[0,1])} &\leq C T^{\frac{\alpha}{2}}\| \bar{u}^{i} - \bar{v}^{i}\|_{C^{\frac{2+\alpha}{2},2+\alpha}([0,T]\times [0,1])}\\
 \left \|  \frac{\bar{v}^{i}_{x}}{|\bar{v}^{i}_{x}|}(\cdot,x)  -\frac{\bar{u}^{i}_{x}}{|\bar{u}^{i}_{x}|} (\cdot, x)\right \|_{C^{0,\frac{1+\alpha}{2}}([0,T])} &\leq C T^{\frac{\alpha}{2}}\| \bar{u}^{i} - \bar{v}^{i}\|_{C^{\frac{2+\alpha}{2},2+\alpha}([0,T]\times [0,1])}
 \end{align*}
 with $C=C(\delta, \|\bar{u}^{i} \|_{C^{\frac{2+\alpha}{2},2+\alpha}([0,T]\times [0,1])}, \|\bar{v}^{i} \|_{C^{\frac{2+\alpha}{2},2+\alpha}([0,T]\times [0,1])} )$.
 \end{lemma}
 \begin{proof}
 It follows by writing every equation in the form $$\frac{a}{|a|}-\frac{b}{|b|}=\frac{1}{|a|}(a-b) + b (\frac{1}{|a|} -\frac{1}{|b|})$$ and using the previous Lemmas~\ref{B2}, \ref{B3}, \ref{lem:hilfsatz}.
 \end{proof}

 \begin{lemma}\label{hilfsatz3}
 Let $h:\R \to \R$ be a smooth map and $u,v \in C^{\frac{\alpha}{2},\alpha}([0,T]\times [0,1])$. Then
 \begin{align*}
  \| h(u) \|_{C^{\frac{\alpha}{2},\alpha}([0,T]\times[0,1])}
  & \leq C \|u \|_{C^{\frac{\alpha}{2},\alpha}([0,T]\times [0,1])}
  \end{align*}
  where $C$ depends on the $C^{1}$-norm of $h$ evaluated on the compact set $K_{1}=u([0,T] \times [0,1])$.
  Similarly
  \begin{multline*}
  \| h(u)-h(v) \|_{C^{\frac{\alpha}{2},\alpha}([0,T]\times[0,1])}\\
  \leq C(1+ \|v\|_{C^{\frac{\alpha}{2},\alpha}([0,T]\times [0,1])}
  + \|u\|_{C^{\frac{\alpha}{2},\alpha}([0,T]\times [0,1])}) \|u -v\|_{C^{\frac{\alpha}{2},\alpha}([0,T]\times [0,1])}
  \end{multline*}
  where $C$ depends on the $C^{2}$-norm of $h$ evaluated on the compact set $$K_{2}=conv(u([0,T] \times [0,1])\cup v([0,T] \times [0,1])).$$
 \end{lemma}
 
 \begin{proof}
 By definition of the norm we have that
 $$\| h(u) \|_{C^{\frac{\alpha}{2},\alpha}([0,T]\times[0,1])} = \sup_{[0,T] \times [0,1]} |h(u(t,x))| +[h(u)]_{\alpha,x} +[h(u)]_{\frac{\alpha}{2},t} $$
 so that, using the mean value, theorem we infer
 $$ \| h(u) \|_{C^{\frac{\alpha}{2},\alpha}([0,T]\times[0,1])} \leq sup_{K_{1}}(|h|+|h'|) (1+ [u]_{\alpha,x} +[u]_{\frac{\alpha}{2},t}) $$
and the first statement follows. The second statement is derived in a similar way. For instance, to estimate $[h(u)-h(v)]_{\alpha,x}$  we compute
 \begin{align*}
 &\frac{|h(u(t,x))-h(v(t,x)) -h(u(t,y))+h(v(t,y))|}{|x-y|^{\alpha}} \\
 &=\frac{|\int_{0}^{1} \frac{d}{d\lambda}[ h(\lambda u(t,x) +(1-\lambda)v(t,x))
 -h(\lambda u(t,y) +(1-\lambda)v(t,y))] d\lambda|}{|x-y|^{\alpha}}\\
 &= \Big|\frac{\int_{0}^{1} h'(\lambda u(t,x) +(1-\lambda)v(t,x))(u(t,x)-v(t,x) ) \,d\lambda}{|x-y|^{\alpha}}\\
 & \quad -\frac{\int_{0}^{1}h'(\lambda u(t,y) +(1-\lambda)v(t,y))(u(t,y)-v(t,y))  \,d\lambda}{|x-y|^{\alpha}} \Big|\\
 & \leq sup_{K_{2}} |h''| ([ u]_{\alpha,x}+ [ v ]_{\alpha,x}) \|u-v\|_{C^{0}} + sup_{K_{2}} |h'| [u-v]_{\alpha,x}  .
 \end{align*}
 \end{proof}

We conclude the Appendix with a repametrization result used in the proof of Proposition \ref{pro:GEO}.

\begin{lemma}\label{lempara}
Let $\mu\in \R$ and $\gamma:[0,L]\to\R^2$ of class $H^3$, with $|\gamma'(x)|=1$ for all $x\in [0,L]$.
We claim that there exists $C=C(L,\mu, \|\gamma\|_{H^3})>0$ and a parametrization $\phi:[0,L]\to [0,L]$ 
such that, 
letting $\tilde\gamma = \gamma\circ\phi$,  it holds
\begin{eqnarray*}
|\tilde\gamma'(x)| &=& \phi'(x)\,\ge\, \frac 12
\qquad \text{for all $x\in [0,L]$,}
\\
\frac{\tilde\gamma''(0)\cdot \tilde\gamma'(0)}{|\tilde\gamma'(0)|^3}  &=&  \frac{\phi''(0)}{\phi'(0)^2} 
\,=\, \mu,
\\ \|\tilde\gamma\|_{C^{2,1/2}}&\le& C.
\end{eqnarray*}
\end{lemma}

\begin{proof}
Let $\delta = \min(L/2,1/(2|\mu|))$ and fix a smooth function $f:[0,L]\to \R$ such that $|f|\le |\mu|$, $f(0)=\mu$,
$f=0$ in $[\delta,L]$ and $\int_0^L f =0$.
We then set $\phi(x)=x+h(x)$, with
\[
h(x) = -\int_x^L \int_0^y f(t) dt dy.
\]
We then have 
\[
h'(x)= \int_0^x f(t) dt,\qquad h''(x)=f(x),
\]
so that $h'(0)=0$, $h'(L)=0$, $h''(0)=f(0)=\mu$, $h''(L)=f(L)=0$ and
\[
|h'(x)|= \left|\int_0^x f(t) dt\right|\le \delta |\mu| \le \frac 12.
\]
Finally we have
$\|\phi\|_{C^{2,1/2}}\le C(L,\|f\|_{C^\frac 12})$,
and the curve $\tilde\gamma= \gamma\circ\phi$ satisfies the required properties.
\end{proof}


\bibliographystyle{acm}

\end{document}